\documentclass[a4paper]{article}

\pdfoutput=1

\usepackage{amsfonts,amssymb,amscd,amsmath,amsthm,enumerate,mathrsfs}
\usepackage[dvips]{graphicx}

\usepackage{graphicx,type1cm,eso-pic}
\usepackage[nottoc]{tocbibind}
\usepackage{hyperref}
\usepackage{fancyhdr}
\newcommand{\abs}{\vspace{12pt}}

\DeclareMathOperator{\Con}{Con}

\DeclareMathOperator{\vol}{vol}

\DeclareMathOperator{\supp}{supp}

\DeclareMathOperator{\id}{id}
\DeclareMathOperator{\Int}{Int}
\DeclareMathOperator{\Clos}{Cl}

\DeclareMathOperator{\const}{const.}

\DeclareMathOperator{\diam}{diam}

\newcommand{\h}{h_{\text{top}}}
\DeclareMathOperator{\skp}{\langle .,. \rangle}
\DeclareMathOperator{\norm}{\|.\|}
\DeclareMathOperator{\la}{\langle}
\DeclareMathOperator{\ra}{\rangle}
\DeclareMathOperator{\es}{es}

\newcommand{\R}{\mathbb{R}}
\newcommand{\Z}{\mathbb{Z}}

\newcommand{\N}{\mathbb{N}}
\newcommand{\MM}{\mathfrak{M}}

\newcommand{\A}{\mathcal{A}}
\newcommand{\U}{\mathcal{U}}
\newcommand{\T}{{\mathbb{T}^2}}
\newcommand{\CC}{\mathcal{C}}
\newcommand{\NN}{\mathcal{N}}

\newcommand{\LL}{\mathcal{L}}
\newcommand{\HH}{\mathcal{H}}
\newcommand{\E}{\mathcal{E}}

\newcommand{\TT}{\mathcal{T}}

\newcommand{\J}{\mathcal{J}}
\newcommand{\s}{\mathscr{S}}
\newcommand{\M}{\mathcal{M}}
\newcommand{\F}{\mathcal{F}}

\newcommand{\G}{\mathcal{G}}

\newcommand{\e}{\varepsilon}
\newcommand{\ep}{\epsilon}
\newcommand{\Om}{\Omega}
\newcommand{\al}{\alpha}
\newcommand{\om}{\omega}
\newcommand{\lam}{\lambda}
\newcommand{\sig}{\sigma}

\swapnumbers

\theoremstyle{plain}
\newtheorem{defn}{Definition}[section]
\newtheorem{lemma}[defn]{Lemma}
\newtheorem{prop}[defn]{Proposition}
\newtheorem{thm}[defn]{Theorem}
\newtheorem{cor}[defn]{Corollary}
\newtheorem*{thm1}{Theorem I}
\newtheorem*{thm2}{Theorem II}
\newtheorem*{thm3}{Theorem III}

\theoremstyle{definition}
\newtheorem{rechnung}[defn]{}
\newtheorem{bemerk}[defn]{Remark}

\newcommand{\mane}{Ma\~n\'e}
\newcommand{\Mane}{Ma\~n\'e }

\newcommand{\Poincare}{Poincar\'e }

\begin{document}

\lhead{} \chead{\footnotesize \textsc{J.P. Schr\"oder -- Tonelli Lagrangians on the 2-torus and topological entropy}} \rhead{} \renewcommand{\headrulewidth}{0pt}

\title{Tonelli Lagrangian systems on the 2-torus  \linebreak
and topological entropy}

\author{Jan Philipp Schr\"oder\footnote{jan.schroeder-a57@rub.de. Faculty of Mathematics, Ruhr-Universit\"at Bochum, Germany.}}

\maketitle

%Invariant tori for Tonelli Lagrangian systems on the 2-torus with vanishing topological entropy

\begin{abstract}
We study Tonelli Lagrangian systems on the 2-torus $\T$ in energy levels $\E$ above \mane's strict critical value. We analyize the structure of global minimizers in the spirit of Morse, Hedlund and Bangert. In the case where the topological entropy of the Euler-Lagrange flow in $\E$ vanishes, we show that there are invariant tori for all rotation vectors indicating integrable-like behavior on a large scale. On the other hand, using a construction of Katok, we give examples of reversible Finsler geodesic flows with vanishing topological entropy, but having ergodic components of positive measure in the unit tangent bundle $S\T$.
\end{abstract}

\renewcommand{\abstractname}{Acknowledgements}
\begin{abstract}
I wish to thank my Ph.D. supervisor Gerhard Knieper for many helpful discussions and support.
\end{abstract}

{\scriptsize \tableofcontents }

\abs\abs

\pagebreak

\section{Introduction and main results}

A Tonelli Lagrangian on a closed manifold $M$ is a smooth function $L:TM\to\R$, such that $L(v)$ grows superlinearly in $\|v\|$ and when restricted to the fibres $T_x M$ has positive definite Hessian, i.e. is strictly convex (on $M$ we fix some Riemannian background metric). The Euler-Lagrange flow $\phi^t$ of $L$ leaves the energy
\[ E(v)=d_vL(v)v-L(v) \]
invariant, so we can restrict the attention to fixed energy levels $\E_k := \{E=k\}$. The main object in this paper is the Euler-Lagrange flow
\[ \phi^t : \E_k \to \E_k \]
for fixed $k\in\R$ in the case where $M=\T=\R^2/\Z^2$ is the 2-torus. We ask if there are $\phi^t$-invariant 2-tori $\TT$ embedded in $\E_k$. In the case where $\phi^t$ is integrable in the sense of Liouville-Arnold it is well known that there are plenty such $\TT$, but generically there need not even be one. A weaker assumption than integrability is that the topological entropy of $\phi^t|_{\E_k}$ vanishes, i.e.
\[ \h(\phi^t,\E_k) = 0. \]

The main result in this paper is the following theorem, which we will prove in section \ref{top ent and invar tori}. We will define $\TT_h, \TT_h^\pm$ and the rotation vectors $h$ below, for information about \mane's critical values, in particular \mane's strict critical value $c_0(L)$, cf. \cite{contreras1}.

\begin{thm1}
Let $L$ be a Tonelli Lagrangian on the 2-torus $\T=\R^2/\Z^2$ with Euler-Lagrange flow $\phi^t$ and let $k>c_0(L)$, where $c_0(L)$ is \mane's strict critical value. If $\h(\phi^t,\E_k) = 0$, then there are $\phi^t$-invariant Lipschitz graphs $\TT_h, \TT_h^\pm \subset \E_k$ over the zero section for all possible rotation vectors $h$. If $v\in\E_k$ does not lie on one of these invariant tori, the orbit $\{\phi^tv\}$ lies in the space enclosed by two tori $\TT_h^-,\TT_h^+\subset\E_k$ of some common rational rotation vector $h$, while these tori intersect in periodic minimizers of rotation vector $h$.
\end{thm1}

\begin{figure}
\centering
\includegraphics[scale=0.9]{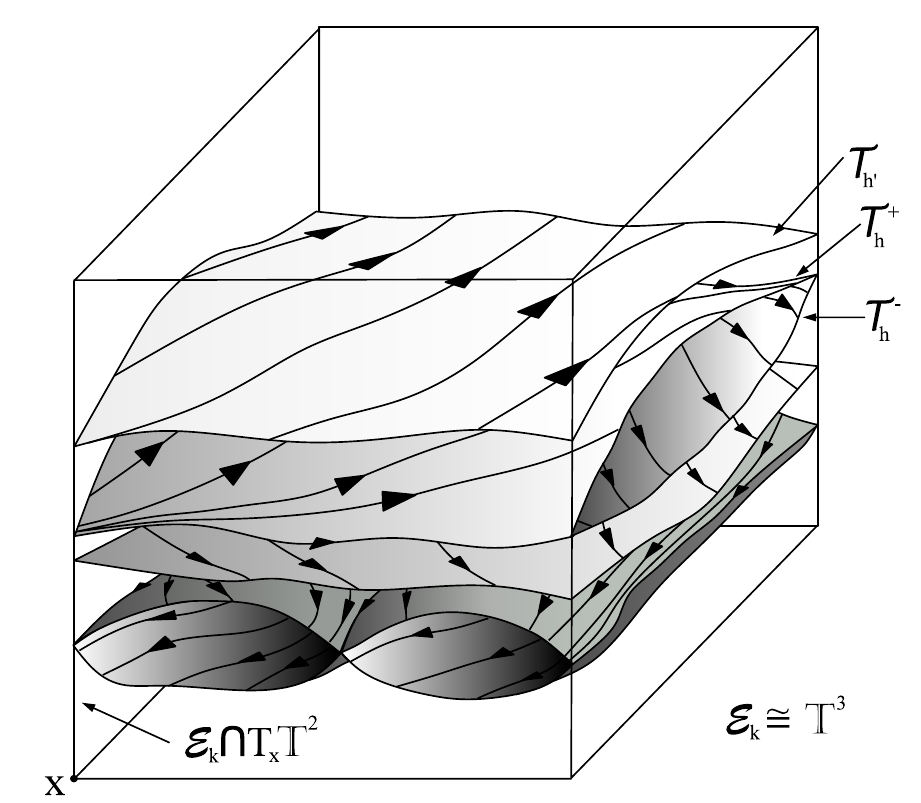}
\caption{The invariant tori $\TT_h,\TT_h^\pm$ in theorem I as graphs over the zero section $0_\T\cong \T$. Note that for $k>c_0(L)$, each $\E_k\cap T_x\T$ is topologically a sphere. Here these spheres are drawn as vertical lines and $0_\T$ can be thought of as the horizontal bottom of $\mathbb{T}^3 \cong \E_k$.}
\end{figure}

Theorem I was previously known in the class of Riemannian metrics due to Glasmachers and Knieper \cite{glasmachers1}. The first result in this direction, covering the class of monotone twist maps, is due to Angenent in 1991, cf. \cite{angenent}. Our approach to prove theorem I is different from that of Angenent, Glasmachers and Knieper, but similar to the techniques of Bosetto and Serra in \cite{bosetto serra}. 

A particular class of Tonelli Lagrangians arises from Finsler metrics, that is a function $F:TM\to[0,\infty)$, such that $F^2$ is strictly convex in the fibres as above, and such that $F$ is positively homogeneous in the fibres, $F(\lam v)=\lam F(v)$ for $\lam\geq 0$. We say that $F$ is reversible, if $F(-v)=F(v)$; if not stated otherwise, Finsler metrics are non-reversible. Here we can only assume that $F$ is smooth outside the zero section $0_M$. Apart from the smoothness issues in $0_M$ (which can be dealt with), the function
\[ L_F := \frac{1}{2}F^2 \]
is a special case of a Tonelli Lagrangian and its Euler-Lagrange flow, also denoted by $\phi^t$, is called its geodesic flow of $F$.

We will generalize the results of Morse \cite{morse}, Hedlund \cite{hedlund} and Bangert \cite{bangert} about minimal geodesics in Riemannian 2-tori to non-reversible Finsler metrics on $\T$. Minimal geodesics are curves $c:\R\to\T$, such that their lifts $\tilde c:\R\to\R^2$ globally realize the $F$-distance, i.e. $l_F(\tilde c[a,b])=d_F(\tilde c(a),\tilde c(b))$ for all $a\leq b$. These minimizers always exist and we shall use them in order to prove theorem I. Hence in section \ref{section mane sets} we will prove the following.

\begin{thm2}
Let $F$ be a Finsler metric (not assumed to be reversible) on the 2-torus $\T$. Then we have the following structure of the set of minimal geodesics. \begin{itemize}
\item[(i)] There is a global constant $D=D(F)\geq0$, such that each lifted minimal geodesic $\tilde c:\R\to\R^2$ has distance at most $D$ from a straight euclidean line in $\R^2$. Write $\delta(c)\in S^1$ for the slope of this line (the orientation of $c$ determines the direction of $\delta(c)$), $S\T=\{F=1\}\subset T\T$ and
\[ \G_\delta := \{ \dot c(0) \in S\T : \text{$\tilde c$ is a minimal geodesic with $\delta(c)=\delta$} \}, \quad \delta\in S^1. \]
The sets $\G_\delta$ are never empty.

\item[(ii)] If $\delta\in S^1$ has irrational slope then $\G_\delta$ is contained in a Lipschitz graph over the zero section, in particular no geodesics from $\G_\delta$ can intersect in $\T$. The set of recurrent vectors for the geodesic flow $\G_\delta^{rec}$ in $\G_\delta$ is a minimal set for the geodesic flow and for each $x\in\R^2$ there are two lifts $\tilde c_0,\tilde c_1$ from $\G_\delta^{rec}$, such that $x$ lies in the strip bounded by $\tilde c_0(\R),\tilde c_1(\R)$ and $\tilde c_0(t),\tilde c_1(t)$ are asymptotic for $t\to-\infty$ and $t\to \infty$. The projection $\pi(\G_\delta^{rec})\subset\T$ is either all of $\T$ or nowhere dense.

\item[(iii)] If $\delta \in S^1$ has rational slope, then $\G_\delta$ contains the non-empty set $\G_\delta^{per}$ of prime-periodic minimal geodesics. Either $\pi(\G_\delta^{per})=\T$ and $\G_\delta^{per}=\G_\delta$ is a Lipschitz graph over the zero section, or $\G_\delta$ decomposes into three non-empty sets $\G_\delta = \G_\delta^{per} \cup \G_\delta^+\cup\G_\delta^-$, where $\G_\delta^\pm$ consist of minimal geodesics heteroclinic to periodic minimal geodesics from $\G_\delta^{per}$. Each of the two sets $\G_\delta^{per} \cup \G_\delta^\pm$ is contained in a Lipschitz graph over the zero section.
\end{itemize}
\end{thm2}

We saw that $\h(\phi^t,\E_k)=0$ implies integrable behavior on a large scale. On the other hand, we construct in section \ref{section katok} the following examples in the fashion of Katok \cite{katok}. The result suggests that, as far as integrable behavior is concerned, the conclusion of theorem I might be optimal.

\begin{thm3}
Let $\Sigma\subset \R^3$ a closed surface of revolution, such that $\Sigma \cap S^2$ contains an open strip around the equator in the standard round sphere $S^2$. Then there exists a reversible Finsler metric $F_0$ arbitrarily close to the standard Finsler metric $\norm_{\R^3}|_{T\Sigma}$ on $\Sigma$ and a smoothly bounded solid torus $Z\subset S\Sigma = \{F_0=1\} \subset T\Sigma$ with non-empty interior and the following properties:
\begin{itemize}
\item[(i)] the geodesic flow of $F_0$ is ergodic in $Z$ (w.r.t. the measure induced by the pullback to $T\Sigma$ of the standard volume form in $T^*\Sigma$ via the Legendre tranform of $L_{F_0}$),
\item[(ii)] there is precisely one periodic geodesic of $F_0$ in $Z$,
\item[(iii)] the topological entropy of $\phi^t_{F_0}:S\Sigma\to S\Sigma$ vanishes.
\end{itemize}
\end{thm3}

\begin{bemerk}
$F_0$ coincides with $\al \cdot \norm_{\R^3}$ for some $\al\approx 1$ in the direction of the meridians, i.e. here there are still plenty of invariant tori. For the case $\Sigma=S^2$ one can make the size of $Z\cup -Z$ arbitrarily large, but we do not destroy all but finitely many periodic orbits.
\end{bemerk}

%%%%%%%%%%%%%%%%%%%%%%%%%%%%%%%%%%%%%%%%%%%%%%%%%%%%%%%%%%%%%%%%%%%%%%%%%%%%%%%%%%%%
%%%%%%%%%%%%%%%%%%%%%%%%%%%%%%%%%%%%%%%%%%%%%%%%%%%%%%%%%%%%%%%%%%%%%%%%%%%%%%%%%%%%
%%%%%%%%%%%%%%%%%%%%%%%%%%%%%%%%%%% definitions %%%%%%%%%%%%%%%%%%%%%%%%%%%%%%%%%%%
%%%%%%%%%%%%%%%%%%%%%%%%%%%%%%%%%%%%%%%%%%%%%%%%%%%%%%%%%%%%%%%%%%%%%%%%%%%%%%%%%%%%
%%%%%%%%%%%%%%%%%%%%%%%%%%%%%%%%%%%%%%%%%%%%%%%%%%%%%%%%%%%%%%%%%%%%%%%%%%%%%%%%%%%%

\section{Definitions, basic properties and notation}\label{section def}

The main assumption in theorem I is sub-exponential complexity of the Euler-Lagrange flow, expressed by $\h(\phi^t,\E_k)=0$. For basics about topological entropy we refer to \cite{walters}. This is the definition we will work with.

\begin{defn}[topological entropy]\label{def top ent}
Let $\phi^t: X\to X$ be a continuous flow in a compact metric space $(X,d)$. Set for $T>0,\e>0$
\[ s(T,\e) := \sup \{ \# S ~|~ S\subset X, \forall x,y\in S, x\neq y ~ \exists t\in[0,T] : d(\phi^tx,\phi^ty)>\e  \} . \]
The \emph{topological entropy of $\phi^t$} can be defined as
\[ \h(\phi^t) = \h(\phi^t,X) = \lim_{\e\to 0}\limsup_{T\to\infty} \frac{\log s(T,\e)}{T}. \]
\end{defn}

Let $M$ be a closed manifold with a Riemannian background metric $\skp$, norm $\norm$ and distance $d$ in $M$ and $TM$, $\pi:TM \to M$ the canonical projection. $0_M\subset TM$ denotes the zero section, $\Int A, \Clos A = \bar A, \partial A = \Clos A- \Int A$ and $\Con_x A$ denote the interior, closure, boundary and connected component of $x\in A$, respectively. A gap of $A$ refers to a connected component of the complement of $A$. For two curves $\gamma_1:[a,b]\to M, \gamma_2:[c,d]\to M$ with $\gamma_1(b)=\gamma_2(c)$ we write
\[ \gamma_1 * \gamma_2: [a,d-c+b]\to M, \quad \gamma_1 * \gamma_2(t):= \begin{cases} \gamma_1(t) & : t\in [a,b] \\ \gamma_2(t+c-b) & : t\in [b,d-c+b]\end{cases}. \]
$C^{ac}(I,M),C^{ac}_{loc}(I,M)$ denote the sets of curves $c:I\to M$ that are absolutely continuous (on compact sets in the latter), endowed with the topologies of $C^0,C^0_{loc}$-convergence. For a Tonelli Lagrangian $L:TM\to\R$ we write
\[ A_L : C^{ac}([a,b],M) \to \R, \quad A_L(c)=A_L(c[a,b])=\int_a^b L(\dot c) dt \]
for the Lagrangian action and
\[ c_v(t) := \pi \phi^tv \]
for the projected flow lines of the Euler-Lagrange flow $\phi^t$ of $L$. For basics on Tonelli Lagrangian systems we refer to \cite{sorrentino} and \cite{contreras}.

\subsection{Mather's and \mane's theories}

In the sequel we will work with the so called Mather theory for Tonelli Lagrangians. We will briefly recall the relevant definitions and set the notation, refering to \cite{sorrentino} for the details. Most of the notation is the same here and there (the biggest difference being that here we write $L_\eta$ for $L-\eta+c(L-\eta)$ instead of $L-\eta$). We fix a Tonelli Lagrangian $L$ on $M$.

\begin{defn}[\mane] \label{def semistatic}
We write
\begin{align*}
\Phi_L(x,y) & = \inf_{T>0} \inf\{ A_L(c[0,T]) : c \in C^{ac}([0,T], M) , c(0)=x, c(T)=y \}, \\
c(L) & = \inf \{ k \in \R:  \forall \text{ closed curves } c : A_{L+k}(c) \geq 0\}
\end{align*}
for \emph{\mane's potential} and \emph{\mane's critical value}. For a closed 1-form $\eta$ on $M$ regarded as a function $\eta:TM\to\R$ we write
\[ L_\eta = L-\eta + \al(\eta) , \qquad \al(\eta) := c(L-\eta) \]
(where $\al:H^1(M,\R)\to\R$ is called \emph{Mather's $\al$-function}) An absolutely continuous curve $c:\R\to M$ is an \emph{$\eta$-semistatic}, if $A_{L_\eta}(c[a,b]) = \Phi_{L_\eta}(c(a),c(b))$ for all $a\leq b$. The \emph{\Mane set of cohomology $[\eta]\in H^1(M,\R)$} is defined as
\[ \NN_\eta = \{ \dot c(0) : c \text{ is $\eta$-semistatic} \} . \]
\end{defn}

\begin{bemerk}
Observe that for a closed 1-form $\eta$ and $k\in\R$ the modified Lagrangians $L-\eta+k$ are still Tonelli and $L,L-\eta+k$ have the same Euler-Lagrange flow. $\al$ and the \Mane set depend only on the cohomology class $[\eta]\in H^1(M,\R)$. $c(L)$ and $\Phi_{L_\eta}(x,y)$ are finite and the \Mane set $\NN_\eta$ is contained in the compact energy level $\{ E = \al(\eta) \}$.\end{bemerk}

We will frequently use the following two semi-continuity properties.

\begin{prop}[lower semi-continuity of the action, p. 174 in \cite{mather}]\label{A semi-cont}
For any $d\in \R$ and any compact set $K\subset \tilde M$, where $\tilde M$ is a covering of $M$, the sets
\[  \{ c\in C^{ac}([a,b], K) : A_L(c ) \leq d \} \]
are compact w.r.t. $C^0$-convergence. In particular, if $c_n\to c$ in $C^0$, we have
\[ A_L(c)\leq \liminf A_L(c_n). \]
\end{prop}

\begin{prop}[upper semi-continuity of the \Mane set w.r.t. $\eta$] \label{N semi-cont}
Let $\eta_n\to \eta$ as closed 1-forms and $v_n\in \NN_{\eta_n}$ with $v_n\to v$ in $TM$. Then $v\in\NN_\eta$.
\end{prop}

We only give the idea, which is standard.

\begin{proof}[Sketch of the proof of \ref{N semi-cont}]
Suppose $v\notin \NN_\eta$, then there are $a<b$ and a curve $\gamma$ with the same endpoints as $c[a,b]$ but having less action w.r.t. $L_\eta$. Up to a small error tending to $0$, $\gamma$ will also have less action than the $c_{v_n}[a,b]$ w.r.t. $L_{\eta_n}\approx L_\eta$, while $\gamma$ (nearly) connects also the endpoints of $c_{v_n}[a,b]$, a contradiction.
\end{proof}

The following definition of $\MM_\eta$ coincides with the original definition of Mather, cf. 5.28 in \cite{sorrentino}

\begin{defn}[Mather]
Let $\MM(L)$ be the set of $\phi^t$-invariant probability measures on $TM$ with compact support, and define the \emph{Mather set of cohomology $[\eta]\in H^1(M,\R)$} by
\[ \M_\eta = \bigcup_{\mu \in \MM_\eta} \supp \mu , \quad \MM_\eta = \{ \mu \in \MM(L) :  \supp \mu \subset \NN_\eta \} . \]
To a measure $\mu\in\MM(L)$ we associate its \emph{rotation vector} $\rho(\mu) \in H_1(M,\R)$ (regarding $H^1, H_1$ as dual spaces with pairing $\skp$) by
\[ \la \rho(\mu) , [\xi] \ra = \int \xi d\mu \quad \forall [\xi]\in H^1(M,\R) \]
(recall $\int \xi d\mu=0$ for exact 1-forms $\xi$). Define \emph{Mather's $\beta$-function} and the \emph{Mather set of a homology class $h$} by
\begin{align*}
& \beta:H_1(M,\R)\to\R, \quad \beta(h) = \inf \{ A_L(\mu) : \mu\in\MM(L),\rho(\mu)=h \} , \\
& \M^h = \bigcup_{\mu \in \MM^h} \supp \mu, \quad \MM^h = \{ \mu \in \MM(L) : \rho(\mu)=h, A_L(\mu)=\beta(h) \}.
\end{align*}
\end{defn}

One has $\M^h\neq\emptyset\neq\M_\eta$ \cite{mather}. The two minimizing procedures are somewhat dual as seen in the following proposition. We denote by $\partial f(x)$ the set of subgradients of a convex function $f$ at $x$. In particular, if $f$ is differentiable at $x$, then $\partial f(x)= \{\nabla f(x) \}$. We refer to \cite{rockafellar} for basics about convex analysis.

\begin{prop}[4.23 and 4.26 in \cite{sorrentino}]\label{mather dual} \begin{itemize}
\item Mather's functions
\[ \al:H^1(M,\R)\to \R, \quad \beta:H_1(M,\R)\to \R \]
are convex, in particular continuous, have superlinear growth and are convex duals of each other.
\item $h\in \partial\al(\eta) \quad \iff \quad \eta\in\partial\beta(h) \quad \iff \quad \M^h \subset \M_\eta$.
\item $\M_\eta = \bigcup_{h \in\partial\al(\eta) } \M^h$.
\item $\al$ is constant on the closed convex sets $\partial\beta(h)$, which we denote by $\F^h$.
\end{itemize} \end{prop}

The last statement can be seen as follows: If $\eta\in \partial\beta(h)$, then $\M^h\subset\M_\eta\subset \{ E=\al(\eta) \}$. Since $\M^h\neq\emptyset$, the claim follows.

\subsection{Fathi's weak KAM theory}

Another approach to the theory of minimizers in Lagrangian systems is Fathi's weak KAM theory. A complete introduction is given in \cite{fathi}, to which we refer for proofs. Fix again a Tonelli Lagrangian $L$ on the closed manifold $M$. In the following we will frequently write 
\[ f |_x^y = f(y)- f(x). \]

\begin{defn}[Fathi] \label{def dominated} \begin{itemize}
\item Let $u:M\to \R$ be a function (no a priori regularity) and $\eta$ a closed 1-form on $M$. We say that $u$ is \emph{(critically) dominated}, written $u\prec L_\eta$, if
\[ u|_x^y \leq \Phi_{L_\eta}(x,y) \qquad  \forall x,y\in M. \]

\item A curve $c\in C_{loc}^{ac}(I,M)$ is said to be \emph{calibrated} on an interval $I\subset\R$ w.r.t. $u$ and $\eta$, if for all $a\leq b$ in $I$ we have
\[ u\circ c|_a^b = A_{L_\eta}(c[a,b]) .\]

\item For $u\prec L_\eta$ write (suppressing $\eta$ in the notation)
\[ \J(u) := \{ \dot c(0) : \text{ $c\in C_{loc}^{ac}(\R,M)$ is calibrated on $\R$ w.r.t. $u,\eta$} \} .\]

\item A function $u\prec L_\eta$ is a \emph{backward weak KAM solution}, if for any $x\in M$ there is a curve $\gamma_x$ calibrated on $(-\infty,0]$ with $\gamma_x(0)=x$. Analogously, a \emph{forward weak KAM solution} has calibrated curves $\gamma_x:[0,\infty)\to M$.
\end{itemize}\end{defn}

Here are some basic results.

\begin{prop}[4.2.1, 4.13.2 and 4.3.8 in \cite{fathi}]\label{calibration} 
Let $u\prec L_\eta$ and $c$ be a curve calibrated on $I\subset\R$ w.r.t. $u,\eta$. Then: \begin{itemize}
\item $u$ is globally Lipschitz continuous, in particular $du$ exists a.e. in $M$.
\item $du$ exists everywhere on $c(\Int(I))$.
\item Let $\LL_{L_\eta} = \LL_L -\eta = \frac{\partial L_\eta}{\partial v}$ be the Legendre transform of $L_\eta$. Then
\[ \forall t\in \Int(I) : \qquad du ( c(t) ) = \LL_{L_\eta} \circ \dot c(t) = \LL_L\circ \dot c(t) - \eta_{c(t)} ~\in T_{c(t)}^*M. \]
\end{itemize} \end{prop}

\begin{thm}[6.20 and 6.21 in \cite{sorrentino}]\label{J(u)}
The sets $\J(u)$ are non-empty, compact and $\phi^t$-invariant. We have
\[ \NN_\eta = \bigcup_{u \prec L_\eta} \J(u), \quad \M_\eta \subset \bigcap_{u \prec L_\eta} \J(u). \]
$\pi$ is injective on the sets $\J(u)$ and $(\pi|_{\J(u)})^{-1}$ is Lipschitz. Moreover let $K\subset H^1(M,\R)$ be compact. Then the Lipschitz constant of $(\pi|_{\J(u)})^{-1}$ can be chosen to be the same for all $u\prec L_\eta, \eta\in K$.
\end{thm}

\subsection{Finsler metrics}

In this paper we will concentrate on Finsler metrics, motivated by the following.

\begin{prop}[cf. \cite{contreras1}]\label{maupertuis}
Let $L$ be a Tonelli Lagrangian on a closed manifold $M$ and $k>c_0(L)$. Then on the energy level $\E_k$, the Euler-Lagrange flow of $L$ is a reparametrisation of a geodesic flow of some Finsler metric on $M$.
\end{prop}

\begin{bemerk}
For $M=\T$, not only Tonelli Lagrangian systems can be described by Finsler geodesic flows, but also the Euler-Lagrange flows arising from non-autonomous Tonelli Lagrangians $L:S^1\times TS^1 \cong \T\times\R\to\R$. One can show that the Finsler metric $F:T\T\to\R$ defined by $F(x,v_1,v_2)=v_1 \cdot L(x,v_2/v_1)$ (where we have to assume that $L(x,r)=F_0(x,1,r)$ for large $|r|$ and some given Finsler metric $F_0$) has the orbits $x(t) = (t,\theta(t))$ of $L$ as reparametrised geodesics. For details we refer to \cite{schroeder}. In particular, using a result of Moser \cite{moser1}, one can study monotone twist maps in the setting of Finsler geodesic flows.
\end{bemerk}

Let $F$ be a Finsler metric on a manifold $M$. As usual, a curve $c\in C^1([a,b],M)$ is said to have arc-length, if $F(\dot c)\equiv 1$. We write
\begin{align*}
& l_F : C^{ac}([a,b],M)\to\R, \quad l_F(c)=l_F(c[a,b])=\int_a^b F(\dot c)dt, \\
& d_F(x,y) = \inf\{ l_F(c) ~:~ c\in C^{ac}([0,1],M), c(0)=x,c(1)=y \}
\end{align*}
for the Finsler length and distance. By homogenity, we have $l_F(c)=l_F(c\circ h)$ for orientation-preserving reparametrisations $h$. Note that the energy of a Finsler Lagrangian $L_F$ is again $E=L_F$. Basics about Finsler metrics and their geodesic flows can be found in \cite{shen}.

A geodesic segment $c:[a,b]\to M$ is said to be minimal, if $l_F( c[a,b]) = d_F( c(a), c(b))$. $c:\R\to M$ is minimal, if each $c|_{[a,b]}$ is minimal. There are two important properties of minimals, which we will use frequently without further notice: \begin{itemize}
\item If $c_i:[0,a_i]\to M, i=0,1$ are arc-length minimals with $c_0(0)=c_1(0)$ and $\dot c_0(0)\neq \dot c_1(0)$, then $c_1(0,a_1]$ is disjoint from $c_0[0,a_0)$.

\item If $c_0,c_1:[0,\infty)\to M$ are arc-length minimals with $c_1(0)=c_0(a)$ for some $a>0$ and $\dot c_0(0)\neq \dot c_1(a)$, then there is some $\e>0$, s.th.
\[ \inf \{ d(c_0(s),c_1(t)) : t\in [a+1,\infty), s\in c_1[1,\infty) \} \geq \e. \]
\end{itemize}
The proofs are the same as in Riemannian geometry, cf. theorems 3 and 6 in \cite{morse}. In the Finsler case orientations of intersecting minimizers matter. We say that two curves $c_0,c_1$ \emph{intersect successively}, if there are times $a_0,a_1,b_0,b_1$, s.th.
\[ c_0(a_0)=c_1(a_1), \quad c_0(b_0)=c_1(b_1), \quad a_0<b_0, \quad a_1<b_1.\]

$F$ and the norm $\norm$ of the background metric are equivalent due to the compactness of $M$. The following technical estimates will be of frequent use, cf. 6.2.1 in \cite{shen}.

\begin{lemma}\label{c_F}
There is a constant $c_F>0$, s.th. the following hold.\begin{itemize}
\item $\frac{1}{c_F} \|v\| \leq F(v) \leq c_F \|v\|$ for all $v\in TM$.
\item $\frac{1}{c_F} d(x,y) \leq d_F(x,y) \leq c_F d(x,y)$ for all $x,y\in M$.
\item $\frac{1}{c_F^2} d_F(y,x) \leq d_F(x,y) \leq c_F^2 d_F(y,x)$ for all $x,y\in M$.
\end{itemize}
\end{lemma}

Using the homogenity of $F$, one easily shows the following.

\begin{lemma}\label{alpha hom}
Let $\al,\beta$ be Mather's functions associated to $L_F=\frac{1}{2}F^2$, $\eta\in H^1(M,\R),h\in H_1(M,\R)$ and $s > 0$. Then
\[ \al(s\eta) = s^2\cdot \al(\eta), \quad \beta(sh) = s^2\cdot \beta(h), \quad \eta\in \partial \beta(h) \iff s\eta\in \partial \beta(sh). \]
\end{lemma}

We want to concentrate on the unit tangent bundle $S M=\{L_F=1/2\}$ and motivated by $\NN_{\eta}\subset\{ L_F=\al(\eta)\}$ we set
\[ H_F := \{\eta \in H^1(M,\R) : \al(\eta)=1/2 \} \subset H^1(M,\R). \]
By lemma \ref{alpha hom}, $H_F$ is star-like w.r.t. $0$ in $H^1(M,\R)$ and bounds the convex set $\{\al \leq 1/2\}$. We also want to single out the homologies that can occur as rotation vectors of the Mather sets $\M^h\subset SM$, so we define another star-like set
\[ G_F := \{ h \in H_1(M,\R) : \partial\beta(h) \subset H_F \} \subset H_1(M,\R) \]
(recall that $\al$ is constant on $\partial\beta(h)$), so for each $h\in G_F$ we have the closed, convex set
\[ \F^h := \partial\beta(h) \subset H_F \subset H^1(M,\R). \]

In the case of $M=\T$, $H_F$ is a circle in $H^1(\T,\R)\cong\R^2$ oriented anticlockwise and the sets $\F^h=\partial\beta(h)$ are straight segments, maybe points in $H_F$. We shall denote them by
\[ \F^h = [\eta_-,\eta_+], \]
where $\eta_\pm$ are the endpoints of $\F^h$ and $\eta_-\leq \eta \leq \eta_+$ for all $\eta\in\F^h$ in the cyclic order defined by the orientation of $H_F$. Similarly we use the orientation in $G_F \subset H_1(M,\R)\cong \R^2$ for a cyclic ordering.

The following calculations show that arc length parametrisation is optimal for $A_{L_\eta}$, if $\eta\in H_F$.

\begin{prop} \label{F=1 optimal} Let $F$ be a Finsler metric on $M$. \begin{itemize}
\item[(i)] For $c\in C^{ac}([0,T],M)$ there is $c_0\in C^{ac}([0,l_F(c)],M)$ with 
\[ F(\dot c_0)=1 \text{ a.e.}, \quad c_0(0)=c(0), \quad c_0(l_F(c)) = c(T), \quad c_0[0,l_F(c)]=c[0,T]. \]
\item[(ii)] Let $k>0$, $\eta$ a closed 1-form on $M$, $c\in C^{ac}([0,T],M)$ and $c_1:[0,\frac{l_F(c)}{\sqrt{2k}}]\to M$ with $c_1(t) := c_0(t \cdot \sqrt{2k})$, $c_0$ as in (i). Then
\[ A_{L_F-\eta+k}(c_1) \leq A_{L_F-\eta+k}(c). \]
\item[(iii)] If $c\in C^{ac}([a,b],M)$ is minimal w.r.t. $A_{L-\eta+k}$, then $c$ is minimal w.r.t. $l_F$ under all curves homologous to $c$.
\end{itemize}\end{prop}

\begin{proof}
(i) Define an equivalence relation $\sim$ on $[0,T]$ by $s\sim t$ iff $l_F(c[0,s])=l_F(c[0,t])$. Identifying $s\sim t$ we obtain a new interval $[0,T_0]\cong [0,T]/_\sim$ and a new curve $\tilde c:[0,T_0]\to M$ defined by $\tilde c([t]):= c(t)$, where $[t]=\{s \in [0,T]: s\sim t\}$. $\tilde c$ is continuous, since $c(t)=c(s)$ for all $s\sim t$. Define $l:[0,T_0]\to [0,l_F(c)]$ by $l([t]):= l_F(c[0,t])$. Then $l$ is continuous, strictly increasing and has a strictly increasing inverse $h:=l^{-1}$. As such, $h$ is differentiable a.e. and for $c_0 := \tilde c \circ h$ we obtain $F(\dot c_0)=h'\cdot F(\dot{\tilde c}\circ h) = \frac{1}{l' \circ h} \cdot F(\dot{\tilde c}\circ h) = 1$ a.e.. It remains to show that $c_0$ is absolutely continuous, but if $\sum_i b_i-a_i \leq \e / c_F$ for $\e>0$, then
\[ \sum_i d(c_0(a_i),c_0(b_i)) \leq c_F \sum_i l_F(c_0[a_i,b_i]) = c_F \sum_i b_i-a_i \leq \e. \]

(ii) From the $L^2$-Cauchy-Schwarz inequality we get
\[ l_F(c[a,b])^2 = \la 1, F(\dot c) \ra_{L^2}^2 \leq \|1\|_{L^2}^2\cdot \|F(\dot c)\|_{L^2}^2 = 2 (b-a) A_{L_F}(c[a,b]) \]
with equality iff $F(\dot c)=\const$ a.e.. By (i) we can reparametrise $c$ to have constant speed (consider the new curve $c_0(t \frac{l_F(c)}{T})$ defined on $[0,T]$, $c_0$ as in (i)) and by Cauchy-Schwarz, the action does not increase. Hence w.l.o.g. $F(\dot c)=\const$ a.e.. For $\lam>0$ consider the curves $c_\lam:[0,T/ \lam]\to M$ with $c_\lam (t)=c(\lam T)$. Then with $E=L_F(\dot c)=\const$ we have
\begin{align*}
A(\lam) := A_{L_F-\eta+k}(c_\lam) = TE\lam - \int_c \eta  + Tk/\lam , \quad A'(\lam) = T\left (E-k/ \lam^2 \right ).
\end{align*}
With $A(\lam)\to \infty$ for $\lam\to 0,\infty$ we find a minimum of $A$ in $(0,\infty)$ characterised by the unique $\lam_0$ with $A'(\lam_0)=0$, i.e. $\lam_0=\sqrt{k/E}$. Thus we find the optimal speed for $c$ by
\[ F(\dot c_{\lam_0}) = \lam_0 \cdot F(\dot c) = \sqrt{2k/F^2(\dot c)} \cdot F(\dot c) = \sqrt{2k}.\]

(iii) Suppose $\gamma:[a,b]\to M$ with $\gamma(a)= c(a), ~ \gamma(b)= c(b)$ is homologous to $c$, $F(\dot \gamma)=\const$ a.e. and $l_F(\gamma) < l_F(c)$. Then we find by the equality case in Cauchy-Schwarz that
\[ A_{L_F+k}(\gamma) = \frac{l_F(\gamma)^2}{2(b-a)} + k(b-a) < \frac{l_F(c)^2}{2(b-a)} + k(b-a) = A_{L_F+k}(c). \]
But $\int_c \eta - \int_\gamma\eta=0$ by $c,\gamma$ being homologous and $\eta$ being closed, contradicting the $A_{L_F-\eta+k}$-minimality of $c$.
\end{proof}

%%%%%%%%%%%%%%%%%%%%%%%%%%%%%%%%%%%%%%%%%%%%%%%%%%%%%%%%%%%%%%%%%%%%%%%%%%%%%%%%%%%%
%%%%%%%%%%%%%%%%%%%%%%%%%%%%%%%%%%%%%%%%%%%%%%%%%%%%%%%%%%%%%%%%%%%%%%%%%%%%%%%%%%%%
%%%%%%%%%%%%%%%%%%%%%%%%%%%%%%%%%%% Mane sets in T^2 %%%%%%%%%%%%%%%%%%%%%%%%%%%%%%%
%%%%%%%%%%%%%%%%%%%%%%%%%%%%%%%%%%%%%%%%%%%%%%%%%%%%%%%%%%%%%%%%%%%%%%%%%%%%%%%%%%%%
%%%%%%%%%%%%%%%%%%%%%%%%%%%%%%%%%%%%%%%%%%%%%%%%%%%%%%%%%%%%%%%%%%%%%%%%%%%%%%%%%%%%

\section{\Mane sets for Finsler metrics on the 2-torus, proof of theorem II}\label{section mane sets}

In this section we study the structure of the \Mane sets $\NN_\eta$ in the case of a Finsler Lagrangian on $\T=\R^2/\Z^2$. From now on we fix a Finsler metric $F$ on $\T$. We will write $p:\R^2\to\T$ for the covering map and $\tau_z=x\mapsto  x+z$ for the deck transformations with $z\in\Z^2$. Our background metric will be the euclidean scalar product $\skp$. We identify $\Z^2$ and $\pi_1(\T)$ with the free homotopy classes of loops in $\T$ and write $c \in z\in \pi_1(\T)$ for closed loops $c$. We will frequently denote objects in $\T$ and lifts to $\R^2$ by the same letters and assume that geodesics are parametrised by $F$-arc-length. When writing $\eta$ we refer to the constant 1-form $\la \eta,.\ra$ on $\T$ with $\eta\in\R^2$. Such $\eta$ also refer the the cohomology classes $[\eta]\in H^1(\T,\R)\cong\R^2$ and we will assume that $\eta\in H_F=\{\al=1/2\}$ if not stated otherwise. Recall that in this case $H_F, G_F$ are oriented circles in $H^1(\T,\R),H_1(\T,\R)\cong \R^2$, respectively.

\subsection{Semistatic curves and rotation vectors}

We use two results due to Hedlund \cite{hedlund} proven in the Riemannian case, just noting that his arguments apply directly to the Finsler case.

\begin{lemma} [lemma 5.1 in \cite{hedlund}] \label{hedlund lemma}
Let $c:\R\to \T$ be a continuous periodic curve with homotopy class $[c]=kz\in\pi_1(\T)$ for some $z\neq 0$ and $k\geq 2$. Then there are $k$ periodic curves $\sig_1,...,\sig_k\in z\in \pi_1(\T)$, s.th. $\sig_1$ is a part of $c=c_0$ and $\sig_{i+1}$ is a part of $c_i$, where $c_i$ is obtained from $c_{i-1}$ by cutting out $\sig_i$, s.th. $c_i\in (k-i)z$.
\end{lemma}

\begin{thm}[Hedlund, Mather] \label{hedlund}
Let $c\in z\in\pi_1(\T)$ be minimal for $l_F$ in the homotopy class $z\neq 0$ and let $T>0$ be the period of $c$. Then \begin{itemize}
\item $c$ is prime-periodic and $c$ is also minimal in $kz$ for all $k\geq 1$,
\item any lift $\tilde c:\R\to\R^2$ is minimal for $l_F$,
\item there is some $\eta\in H_F$ with $\dot c\in \M_\eta$.
\end{itemize}
\end{thm}

The first two statements are contained in Hedlund's paper \cite{hedlund}, while the last statement is a special case of proposition 2 in Mather's paper \cite{mather}.

\begin{defn}\label{def rotation vector}
Let $c:\R\to \T$ and $\tilde c:\R\to\R^2$ be a lift. Set
\begin{align*}
\delta^\pm(c) &:= \lim_{T\to\pm\infty}\frac{\tilde c(T)}{\|\tilde c(T)\|}  ~ \in S^1, \\
\rho(c) &:= \lim_{T\to \infty}\frac{\tilde c(T)-\tilde c(-T)}{|2T|}, \quad \rho^\pm(c) := \lim_{T\to\pm\infty}\frac{\tilde c(T)}{|T|} \in \R^2 ,
\end{align*}
if the limits exist. We call $\delta^\pm$ \emph{asymptotic directions} and $\rho,\rho^\pm$ \emph{rotation vectors}.
\end{defn}

\begin{rechnung} \label{<eta,rho>=1}
We make a few basic observations. Let $c$ be an $\eta$-semistatic ($\eta\in H_F$). \begin{itemize}
\item There exists a global constant $C$, s.th. $|A_{L_\eta}(c[a,b])|\leq C$ for any $a<b$. This follows from comparison with minimal geodesics between $c(a),c(b)$ in $\T$, giving a global bound for $\Phi_{L_\eta}(x,y)$. In particular, if $c$ is periodic with period $T$, we have $A_{L_\eta}(c[0,T])=0$: $A_{L_\eta}(c)\geq 0$ by definition of \mane's critical value and if $A_{L_\eta}(c)>0$, the action would become unbounded for higher iterates of $c$.

\item Calculating $\lim_{T\to\infty} \frac{1}{T}A_{L_\eta}(c[0,T])=0$, we find by $L_F(\dot c)=\al(\eta)=1/2$, that
\[ \frac{1}{T} \int_0^T \la \eta, \dot c \ra dt \to 1, \quad T\to\infty. \]
In particular, if $\rho(c),\rho^+(c)$ exist, we find by $\int_a^b\la \eta, \dot c\ra dt= \la\eta,\tilde c(b)-\tilde c(a)\ra$, that
\[ \la\eta,\rho(c)\ra=\la\eta,\rho^+(c)\ra=1. \]
This refelects that $\rho$ depends on the parametrisation fixed by $\al(\eta)$.

\item For the lifts $\tilde c$ of $c$ we have $\lim_{t\to \pm\infty}\|\tilde c(t)\|=\infty$ due to minimality. In particular, $\delta^\pm$ do not depend on the chosen lift.
\end{itemize}
\end{rechnung}

\begin{thm}\label{existence rot vector} For any $\eta$-semistatic the asymptotic directions and rotation vectors exist and
\[ \delta^+=-\delta^-, \quad \rho=\rho^+=-\rho^-, \quad \rho = \frac{\delta^+}{\la \eta,\delta^+ \ra}, \quad  \delta^+ = \frac{\rho}{\|\rho\|}. \]
Furthermore all $\eta$-semistatics have the same rotation vector.
\end{thm}

\begin{proof}
Let $c$ be an $\eta$-semistatic lifted to $\R^2$. We first prove everything for $\delta^\pm$. Recall that $\|c(t)\|\neq 0$ for large $|t|$ by the observations \ref{<eta,rho>=1} and suppose $\delta^+(c)$ would not exist. Then we find two limit points of $\frac{c(t)}{\|c(t)\|}$ for $t\to\infty$, say $v_1=\lim\frac{c(s_i)}{\|c(s_i)\|},v_2=\lim\frac{c(t_i)}{\|c(t_i)\|}$ with $s_i<t_i<s_{i+1}$. Put two disjoint open cones $C_j$ around $\R_{>0}v_j$. For large $i$ we have $c(s_i)\in C_1$ and $c(t_i)\in C_2$ (as long as e.g. $v_1\approx \frac{c(s_i)}{\|c(s_i)\|} \in C_1$), so $c$ oscillates between $C_1,C_2$. This shows that for some periodic minimal $c_0$ from theorem \ref{hedlund}, $c$ has to intersect $c_0$ successively, contradicting the minimality of $c,c_0$.

To show $-\delta^-(c)= \delta^+(c)$ suppose the contrary and consider disjoint open cones $C_-,C_+$ around $\R_{>0}\delta^-(c)$, $\R_{>0}\delta^+(c)$, respectively. By the convergence $\frac{c(t)}{\|c(t)\|}\to\delta^\pm$, we have $c(\pm t)\in C_\pm$ for large $|t|$. Also recall $\|c(t)\|\to\infty, t\to\pm\infty$. Now consider a periodic minimal $c_0$ intersecting first $C_-$, then $C_+$ far away from the origin. This again gives successive intersections of $c,c_0$.

To show the existence of $\rho^+(c)$ observe that by \ref{<eta,rho>=1}
\begin{align*}
& \underset{\to 1}{\underbrace{\frac{1}{T} \int_0^T \la \eta, \dot c \ra dt}} = \left \langle \eta , \frac{c(T)-c(0)}{\|c(T)\|}\frac{\|c(T)\|}{T} \right \rangle = \frac{\|c(T)\|}{T} \left \langle \eta , \underset{\to \delta^+(c)}{\underbrace{\frac{c(T)-c(0)}{\|c(T)\|}}} \right \rangle \\
\Rightarrow \quad & \frac{c(T)}{T} = \frac{\|c(T)\|}{T}\frac{c(T)}{\|c(T)\|} \to \frac{\delta^+(c)}{\la \eta,\delta^+(c) \ra} = \rho^+(c) .
\end{align*}

With $-\delta^-=\delta^+$ we get with the same calculations as above for $\delta^-$:
\[ \rho^- = \lim_{T\to \infty} \frac{c(-T)}{T} = \frac{\delta^-}{\la \eta, -\delta^- \ra} = -\frac{\delta^+}{\la \eta, \delta^+ \ra} = -\rho^+. \]
Finally
\[ \rho^+ = \frac{1}{2}(\rho^+(c)-\rho^-(c) ) = \lim_{T\to\infty}\frac{c(T)}{2T}-\frac{c(-T)}{2T} = \rho . \]

To show the uniqueness of $\rho$ in $\NN_\eta$, choose some curve $c_0$ from $\M_\eta$. By theorem \ref{J(u)} we find some $u\prec L_\eta$, s.th. $c,c_0$ are calibrated for $u$ ($c_0$ is calibrated for all $u\prec L_\eta$) and $c,c_0$ cannot intersect in $\R^2$. This shows $\delta^+(c) = \pm \delta^+(c_0)$ (otherwise consider two disjoint cones $C,C_0$ around $\R\delta^+(c),\R\delta^+(c_0)$, now $c$ has to get from one side of $C_0$ to the other). Now
\[ \rho(c) = \frac{\delta^+(c)}{\la \eta, \delta^+(c) \ra} = \frac{\pm \delta^+(c_0)}{\la \eta, \pm \delta^+(c_0) \ra} = \rho(c_0). \]
\end{proof}

\begin{cor}\label{alpha diffbar}\begin{itemize}
\item $\al$ is $C^1$ everywhere in $H^1(\T,\R)-\{0\}$.
\item $\M_\eta=\M^{\nabla\al(\eta)}$ for all $\eta\in H^1(\T,\R)-\{0\}$.
\item The rotation vector of $\eta$-semistatics is $\rho=\nabla\al(\eta)$.
\item $\beta$ is strictly convex.
\end{itemize}\end{cor}

\begin{proof}
Let $\xi$ be a closed 1-form on $\T$ and $\mu\in\MM_\eta$. All $v\in\NN_\eta$ have some fixed rotation vector $h=\rho(c_v)$ by theorem \ref{existence rot vector}, so
\[ \la\xi, h\ra = \la\xi,\rho^+(c_v)\ra = \lim_{T\to\infty} \frac{1}{T}\int_0^T \la\xi,\dot c_v(t)\ra dt =: \hat \xi(v) , \]
i.e. $\hat\xi$ is the time avarage of $\xi:T\T\to\R$, it exists everywhere and is constant. Using Birkhoff's ergodic theorem, we obtain
\[ \la\xi, h\ra \stackrel{\mu(TM)=1}{=} \int_{TM} \hat \xi d\mu =  \int_{TM} \la\xi,v\ra d\mu(v) \stackrel{\text{def $\rho(\mu)$}}{=} \la \xi,\rho(\mu) \ra \quad \Rightarrow\quad h=\rho(\mu). \]
Hence, all $\mu\in \MM_\eta$ have the fixed rotation vector $h$. By proposition \ref{mather dual}, $\#\partial\al(\eta)=1$ and $\M_\eta=\M^{\nabla\al(\eta)}$ follow and in particular $h=\nabla\al(\eta)$. By theorem 24.4 in \cite{rockafellar}, $\al$ is $C^1$. The strict convexity of $\beta$ is a consequence of proposition 4.27 (i) in \cite{sorrentino} and proposition \ref{mather dual}.
\end{proof}

By homogenity of $\al$ one sees that $H_F$ is a $C^1$-submanifold of $\R^2$ (implicit funtion theorem) bounding the convex region $\{\al\leq 1/2\}$. This shows the following.

\begin{cor}\label{alpha C^1}
Consider the map $\rho: H_F \to G_F$ with $\eta\mapsto \nabla \al(\eta)$. Then \begin{itemize}
\item[(i)] the lifts $\tilde \rho:\R\to\R$ are non-decreasing (recall $H_F,G_F\cong S^1$),
\item[(ii)] $\rho$ is surjective,
\item[(iii)] $\rho$ has mapping degree 1.
\end{itemize}
\end{cor}

In the following two subsections we study the structure of the \Mane sets $\NN_\eta$. For $u\prec L_\eta$ the invariant set $\J(u)$ is contained in a Lipschitz graph and curves $c_v(\R),c_w(\R)\subset\R^2$ with $v,w\in\J(u)$ are equal or disjoint. Suppose each such geodesic would intersect the verticals $V_n = \{x \in\R^2: x_1=n\}, n\in\N,$ exactly once. We could then describe the geodesic flow in $\J(u)$ as the \Poincare map of the projected geodesic flow in $\pi(\J(u))$ from $V_0$ to $V_1$. Interpolating this map in the gaps linearly and projecting to the torus would give a circle homeomorphism with rotation number linked to $\rho(\J(u))=\rho(\eta)$. This approach is carried out in Bangert's article \cite{bangert}. Our techniques are as well motived by the study of circle homeomorphisms. The distinction between rational or irrational slope of $h=\rho(\eta)$ becomes fundamental (we will just say that $h$ is (ir)rational).

In the following we work in $\R^2$ and consider all objects as lifted, using the same letters. Let $v\in\NN_\eta\subset T\R^2$ and choose $u\prec L_\eta$ with $v\in\J(u)$. By theorem \ref{existence rot vector}, $c_v(\R)$ has sublinearly bounded distance from $\R h$. We can define a partial ordering on $\J(u)$.

\begin{defn}\label{ordering}
For $v\in\NN_\eta$ let $V(v-)\subset\R^2$ be the closed half space below and $V(v+)$ the one above $c_v(\R)$ (here we pick the orientation on $(\R h)^\perp$ defined by $i\cdot h$). Define for $v,w\in \J(u)\subset T\R^2$
\[ v < w \quad :\iff \quad  c_v(\R) \subset \Int V(w-) \quad \iff \quad c_w(\R)\subset \Int V(v+). \]
\end{defn}

\begin{bemerk}
Obviosly $v<w$ iff $\phi^sv<\phi^tw$ for any $s,t\in\R$. We also write $c_v<c_w$.
\end{bemerk}

\subsection{Irrational directions}

Let $h\in G_F$ have irrational slope and fix $\eta\in \F^h$. The following arguments mimic arguments for circle homeomorphisms with irrational rotation number, cf. also chapter 3 in \cite{denzler} or chapters 2 and 4 in \cite{bangert}. As for irrational circle homeomorphisms, the $\al$- and $\om$-limit sets play a crucial role in the study of the dynamics in $\NN_\eta$.

\begin{defn} \label{limit sets}
For $v\in\NN_\eta$ set $R(v)=dp^{-1}(\al(v)\cup\om(v))$, i.e. $R(v)$ is the lift of the $\al$-/$\om$-limit sets of $v$ w.r.t. the geodesic flow $\phi^t$.
\end{defn}

\begin{bemerk}\label{bem limit sets}\begin{itemize}
\item[(i)] It is well known that $R(v)$ is closed and $\phi^t$-invariant. If $v\in\J(u)$ for some $u\prec L_\eta$, then by the closedness and $\phi^t$-invariance of $\J(u)$ we have $R(v)\subset \J(u)$.
\item[(ii)] $R(v)$ could be analogously defined as the set of $w\in S\R^2$, s.th. there are times $t_i\in\R$ and translates $z_i\in\Z^2 -\{ 0\}$ with $w=\lim_i d\tau_{z_i}\phi^{t_i} v$.
\end{itemize}\end{bemerk}

Here are some properties of $R(v)$.

\begin{prop} \label{eigenschaften M} Let $R=R(v)$. Then \begin{itemize}
\item[(i)] for $w\in R$ we have $R(w)=R$,
\item[(ii)] $dp(R)\subset S\T$ is minimal for $\phi^t$ (no non-trivial closed invariant subsets), 
\item[(iii)] if $\pi (R)\subset \R^2$ has non-empty interior, then $\pi (R) = \R^2$.
\end{itemize}
\end{prop}

For the proof we need the following lemma, which is a consequence of the irrationality of $h$, by which $\NN_\eta$ cannot contain periodic orbits.

\begin{lemma} \label{h irrational}
Let $v\in \J(u)$ be fixed. For $z\in\Z^2-\{0\}$ there is $\e(z)>0$, s.th. in $\R^2$ the curves $c_v(\R)$ and $\tau_z c_v(\R)$ stay at distance $\e(z)$ from another.
\end{lemma}

\begin{proof}[Proof of \ref{h irrational}]
Suppose there are $t_k,\tilde t_k\in\R$ with $\e_k := d_F(c_v(t_k),\tau_zc_v(\tilde t_k))\to 0$ monotonic. We find $z_k\in\Z^2$, s.th. w.l.o.g. $\tau_{z_k}c_v(t_k)$ and hence $\tau_{z_k}\tau_zc_v(\tilde t_k)$ converge to some $x\in\R^2$ and both velocity vectors of the two curves converge to some $v_0\in \J(u) \cap S_x\T$ by the graph property. By minimality of $c_v$ in $\R^2$ and $l_F(c_v[a,b])=b-a$, we have
\[ \tilde t_k \leq d_F(\pi v, \tau_z \pi v) + t_k + \e_k , \qquad  t_k \leq d_F(\pi v, \tau_z \pi v) + \tilde t_k + \e_k, \]
i.e. w.l.o.g. $\tilde t_k-t_k \to T$ for some $T\in\R$. With $\phi^t d\tau_z=d\tau_z\phi^t$ ($\tau_z$ being isometries w.r.t. $F$) we have
\[ \phi^T v_0 = \lim_k \phi^{\tilde t_k-t_k} d\tau_{z_k} \phi^{t_k} v = \lim_k d\tau_{z_k} \phi^{\tilde t_k} v = d\tau_z^{-1}\lim_k d\tau_{z_k}d\tau_{z} \phi^{\tilde t_k} v = d\tau_z^{-1} v_0 , \]
giving a periodic orbit $\dot c_{v_0}$ in $\J(u)$, contradiction.
\end{proof}

\begin{proof}[Proof of \ref{eigenschaften M}]
(i) Let $w\in R$. As $R$ is closed and invariant, we have $R(w)\subset R$. We prove $R\subset R(w)$. By remark \ref{bem limit sets}, we find $z_i, t_i$ with $d\tau_{z_i} \phi^{t_i}v\to w$ and hence for any given $\e>0$ some $z(\e)\in\Z^2$ s.th.
\[ d_{(\R^2,F)}(\tau_{z(\e)} c_w(0),c_v(\R))\leq \e/2. \]
Using lemma \ref{h irrational} and taking any $w_0=\lim_{i\to\infty}d\tau_{z_i^0} \phi^{t_i^0}v\in R$ with distinct $z_i^0\in\Z^2$ we set
\[ \e_i:=\min\{ \e(z_i^0-z_{i-1}^0), \e(z_{i-1}^0-z_i^0) \}>0. \]
Now $\tau_{z(\e_i)}c_w(\R)$ lies between $c_v(\R)$ and one of the curves
\[ \tau_{z_i^0-z_{i-1}^0}c_v(\R), \quad \tau_{z_{i-1}^0-z_i^0}c_v(\R) \]
by lemma \ref{h irrational}. W.l.o.g. the approximation $d\tau_{z_i^0} \phi^{t_i^0}v\to w_0$ is strictly monotone w.r.t. $<$, say, increasing and also we may assume that all $\tau_{z(\e_i)}c_w(\R)$ lie on one side of $c_v(\R)$, say $\tau_{z(\e_i)}c_w(\R) \subset V(v-)$. Then $\tau_{z_i^0}\tau_{z(\e_i)}c_w(\R)$ lies between $\tau_{z_{i-1}^0}c_v(\R)< \tau_{z_i^0} c_v(\R) \subset V(w_0-)$, and by the approximation $\tau_{z_{i-1}^0}c_v(\R)\to c_{w_0}(\R)$ we can also approximate $w_0$ by $w$.

(ii) Any closed invariant set $N\subset R$ with $w\in N$ contains $R(w)=R$ by (i).

(iii) The closed set $R$ is contained in the Lipschitz graph $\J(u)$ over $0_{\R^2}$, so $\pi(R)\subset \R^2$ is closed. Let $U\subset\pi (R)$ be some non-empty open set. For $w\in R$ we find some time $t$ with $c_w(t)\in U$ by (ii), so for $\tilde U=(\pi|_{\J(u)})^{-1}(U)\subset R$, the set $\pi\phi^{-t}(\tilde U)$ is an open set in $\pi(R)$ containing $\pi w$. Thus $\pi(R)$ is also open, so $\pi(R)=\R^2$.
\end{proof}

For irrational circle homeomorphisms it is well known that the limit set $\om(x)\subset S^1$ is independent of $x$. Here is the corresponding result for $R(v)\subset \NN_\eta$.

\begin{prop} \label{M eindeutig}
For all $v,w\in\NN_\eta$ we have $R(v)=R(w)=:R$.
\end{prop}

\begin{proof}
Let $v\in\NN_\eta$ and $w\in\M_\eta$, then $v,w$ are contained in a common graph $\J(u)$ (cf. theorem \ref{J(u)}). Set
\[ e := \inf\{ d_F(\tau_z c_w(t),c_v(\R)) : z\in\Z^2, t\in \R \} \geq 0. \]
We claim $e=0$. Assume $e>0$ and consider the closed strip $S\subset \R^2$ defined by
\[ \left (\cap \{ V(d\tau_z w -) : z\in \Z^2, d\tau_z w>v \} \right ) \bigcap \left (\cap \{ V(d\tau_z w +) : z\in \Z^2, d\tau_z w< v \} \right ). \]
$S$ contains $c_v(\R)$ in its interior by $e>0$ and is bounded by two geodesics $c_\pm$ from $R(w)$, s.th. $c_-<c_v<c_+$. By definition of $e$ and the irrationality of $h$, $S$ is disjoint from all its translates. Hence $p: S\to p(S)\subset \T$ is injective, while $p(S)$ has uniform width $\geq e>0$, contradicting $\vol_{euc}(\T)<\infty$. But for $e=0$ we find $s_i^\pm,t_i\in \R$, s.th.
\[ \min_{+,-} \lim_{i\to\infty} d( c_\pm(s_i^\pm) , c_v(t_i)) = 0.\]
Using the graph property and applying suitable translates, we find $R(w)\cap R(v)\neq \emptyset$ by $\dot c_\pm \in R(w)$. It now follows from the minimality of $R(v),R(w)$ that $R(v)=R(w)$. This shows the claim, as for $v'\in\NN_\eta$ we have $R(v')=R(w)=R(v)$.
\end{proof}

\begin{cor} \label{M mather-menge} \begin{itemize}
\item[(i)] For any $\eta\in \F^h$ we have
\[ dp(R)=\M^h=\{v\in\NN_\eta: \text{$v$ is recurrent} \}. \]
\item[(ii)] For $x\in \R^2$ there are two (unique) geodesics $c_-,c_+$ from $R$ closest to $x$, s.th. $x$ lies in the strip between $c_\pm$ and $c_\pm$ are asymptotic in $\infty$ and $-\infty$.
\end{itemize} \end{cor}

\begin{proof}
(i) Picking $v\in \M^h$ we find $dp(R)=dp(M(v))\subset\M^h$ by the closedness and invariance of $\M^h$. Since $\M^h=\M_\eta\subset \NN_\eta$ consists of supports of invariant measures, any point in $\M^h$ is recurrent. Finally, if $v\in\NN_\eta$ is recurrent, then by definition $v\in dp(R(v))=dp(R)$.

(ii) The existence of the $c^\pm$ closest to $x$ follows from the closedness of $\pi(R)$. If they were not asymptotic, say in $+\infty$, the same argument as in proposition \ref{M eindeutig} gives a contradiction to $\vol_{euc}(\T)<\infty$. 
\end{proof}

The following proposition shows that there is only one \Mane set for irrational rotation vectors. Bangert proved an analogous statement for discrete variational problems in \cite{bangert2} using similar techniques. For monotone twist maps the result is due to Mather \cite{mather1}.

\begin{prop} \label{beta diffbar irrational}
$\beta$ is differentiable in irrational directions.
\end{prop}

\begin{proof}
Let $\eta_0,\eta_1\in\F^h$ and $u_i\prec L_{\eta_i}$ be backward weak KAM solutions. By theorem \ref{J(u)}, the curves from $\M^h=\M_{\eta_i}$ belong to the $(u_i,\eta_i)$-calibrated curves for both $i=0,1$. Let $\gamma_0:(-\infty,0]\to\R^2$ be a $(u_0,\eta_0)$-calibrated curve. By definition for $t<0$ there is a $(u_1,\eta_1)$-calibrated curve $\gamma_1:(-\infty,0]\to \T$ with $\gamma_1(0)=\gamma_0(t)$. Both $\gamma_i$ lie in one common gap of $\pi(\M^h)$ and thus by corollary \ref{M mather-menge} are asymptotic in $-\infty$ (they also have the same rotation vector $h$). This shows that $\gamma_1$ is just a reparametrisation of $\gamma_0$ (they cannot intersect transversely by minimality in $\R^2$). Hence, using the continuity of $u_1,t\mapsto A_{L_{\eta_1}}(\gamma_0[a,t])$, $\gamma_0$ is also $(u_1,\eta_1)$-calibrated on $(-\infty,0]$, so both $u_i$ have the same calibrated curves on $(-\infty,0]$. Let $U\subset\R^2$ be the set, where both $u_i$ are differentiable (this is a set of full Lebesgue measure by Rademacher's theorem). By 4.13.2 in \cite{fathi}, $du_i$ are continuous in $U$. For $x$ in $U$ choose a calibrated curve $\gamma_x:(-\infty,0]\to\R^2$ with $\gamma_x(0)=x$. With proposition \ref{calibration} we have
\[ (du_i+\eta_i)(x) = \lim_{t\nearrow 0} (d u_i + \eta_i)(\gamma_x(t)) = \lim_{t\nearrow 0} \LL_L (\dot\gamma_x(t)) = \LL_L (\dot\gamma_x(0)) \]
independently of $i=0,1$, so $u_1+\eta_1 - (u_0+\eta_0)=:c$ is constant on $\R^2$ (the derivative vanishes a.e. observing $\eta_i=d\eta_i$). This shows that
\[ u_1-u_0-c = \eta_0-\eta_1. \]
The left hand side of this equation is periodic, hence bounded and the right hand side is a linear function on $\R^2$. Therefore $\eta_0-\eta_1$ has to vanish.
\end{proof}

\begin{cor} \label{semistatic => static}
For irrational $h$ the $\nabla\beta(h)$-semistatics are calibrated on $\R$ for all $u\prec L_{\nabla\beta(h)}$, i.e. $\NN_{\nabla\beta(h)} = \J(u)$.
\end{cor}

\begin{proof}
Let $\eta=\nabla\beta(h)$, $v\in\NN_\eta$ and $c$ be a geodesic from $\M^h$ closest to $c_v$. Choose $u\prec L_\eta$ and sequences $s_n,\tilde s_n\to-\infty, t_n,\tilde t_n\to +\infty$, s.th. on $\T$ we have convergence $\lim c(s_n),\lim c_v(\tilde s_n) = x, \lim c(t_n)=\lim c_v(\tilde t_n) = y$ (this is possible, since $c,c_v$ are asymptotic). We have with $\M^h\subset\J(u)$ that
\begin{align*}
 u|_x^y & = u\circ c_v|_a^b + \lim_{n\to \infty} u\circ c_v|_{\tilde s_n}^a + u\circ c_v|_b^{\tilde t_n} \\
& \leq A_{L_\eta}(c_v[a,b])  + \lim_{n\to \infty} A_{L_\eta}(c_v[\tilde s_n,a]) + A_{L_\eta}(c_v[b,\tilde t_n])  \\
& = \lim_{n\to \infty} A_{L_\eta}(c_v[\tilde s_n,\tilde t_n]) = \Phi_{L_\eta}(x,y)  = \lim_{n\to\infty} A_{L_\eta}(c[s_n,t_n])  = u|_x^y,
\end{align*}
i.e. by $u\prec L_\eta$ we obtain $u\circ c_v|_a^b = A_{L_\eta}(c_v[a,b])$ for all $a\leq b$.
\end{proof}

We summarize our results for irrational $h\in G_F$.

\begin{thm} \label{irrational directions}
Let $h\in G_F$ have irrational slope. Then: \begin{itemize}
\item $\beta$ is differentiable in $h$, in particular, there is only one \Mane set corresponding to $h$.
\item The \Mane set $\NN_{\nabla\beta(h)}$ is equal to the \emph{Aubry set} $\A_{\nabla\beta(h)}:=\cap_{u\prec L_{\nabla\beta(h)}} \J(u)$ and in particular a Lipschitz graph over $0_\T$.
\item The Mather set $\M^h=\M_{\nabla\beta(h)}$ is minimal for $\phi^t$ and is equal to the set of recurrent points in $\NN_{\nabla\beta(h)}$.
\item $\pi(\M^h)\subset\T$ is either all of $\T$ or nowhere dense.
\item For $x\in \R^2$ there are two (unique) geodesics $c_-,c_+$ from $\M^h$ closest to $x$, s.th. $x$ lies in the strip between $c_\pm$ and $c_\pm$ are asymptotic in $-\infty$ and in $\infty$. In particular, any orbit in $\NN_{\nabla\beta(h)}$ is homoclinic to $\M^h$.
\end{itemize}\end{thm}

\subsection{Rational directions}

The following theorem is an analogue of the classification of orbits for circle homeomorphisms with rational rotation number. The general ideas in the proof have been present for a long time, cf. for instance theorems 10 and 14 in the classical paper of Morse \cite{morse}.

\begin{thm} \label{rational directions}
Let $h\in G_F$ have rational slope and $\F^h=[\eta_-,\eta_+]$.\begin{itemize}
\item[(i)] All $c_v$ with $v\in\NN_\eta, \eta\in\F^h$ are either prime-periodic or heteroclinic connections between two neighboring periodic minimizers.
\item[(ii)] $\M^h$ consists of the periodic minimizers:
\begin{align*}
\M^h & = \left \{ v\in\cup_{\eta\in\F^h}\NN_\eta: \text{ $c_v$ periodic } \right \}.
\end{align*}
\item[(iii)] Either $\pi(\M^h)=\T$ or in each gap between two neighboring periodic minimizers, there exist geodesics $c_\pm$ with $\dot c_\pm\in\NN_{\eta_\pm}$, s.th. $c_\pm$ are heteroclinics between the periodics and in $\R^2$, $c_{v_+}$ approaches the lower (w.r.t. $<$) periodic minimal in $-\infty$ and the upper in $+\infty$; $c_-$ has the opposite behavior.
\end{itemize}\end{thm}

\begin{bemerk}\label{bem rational directions}
In proposition \ref{beta diffbar irrational}, we proved that $\beta$ is always differentiable in irrational directions. This is not true in rational directions, and in fact we have for rational $h$, that
\[ \# \partial\beta(h) = 1 \quad \iff \quad \pi(\M^h) = \T, \]
i.e. $\beta$ is differentiable in rational $h\in G_F$ iff $\T$ is foliated by periodic minimals of direction $h$. For this result cf. proposition 1 in \cite{massart sorrentino}.
\end{bemerk}

\begin{proof}
(i). By the graph property in theorem \ref{J(u)}, an $\eta$-semistatic $c$ lifted to $\R^2$ cannot intersect its translates or the lifted periodic minimals (they belong to $\M^h\subset \J(u)$ for any $u\prec L_\eta,\eta\in\F^h$ by theorem \ref{hedlund}). Let $\tau$ be the prime translation associated to $h$ and assume $\tau c(\R)\neq c(\R)$. As periodic minimizers are prime-periodic we obtain family $\{\tau^k c(\R)\}_{k\in\Z}$ of disjoint curves in $\R^2$, ordered w.r.t. $<$ and contained in the strip between to periodic minimizers. Hence we have limits (geodesics) $q_0=\lim_{k\to +\infty}\tau^k c$ and $q_1=\lim_{k\to -\infty}\tau^k c$. Obviously $\tau q_i=q_i$, so (i) follows.

(ii) Let $P$ be set of periodic geodesics in $\cup_{\eta\in\F^h}\NN_\eta$, then we know $P\subset \M^h$ by theorem \ref{hedlund}. If $v\in \M^h$, then $v\in \NN_\eta$ by $\M^h=\M_\eta$ for any $\eta\in\F^h$. But because $v$ lies in the support of some $\mu\in\MM(L)$, $v$ is recurrent and, since by (i) the only recurrent geodesics in $\M^h$ are periodic, $c_v$ is periodic (if there is only one periodic geodesic, we move to a finite cover of $\T$, so heteroclinics are not recurrent), hence $v\in P$.

(iii) We have to construct the heteroclinic $c_+$. Let $q_0<q_1$ be neighboring periodics in $\R^2$ from $\M^h$ (i.e. there is no other curve $q_2$ from $\M^h$ with $q_0<q_2<q_1$). Take some sequence $h_i\searrow h$ in $G_F$ and $v_i\in \M^{h_i}$. The curves $c_{v_i}$ in $\R^2$ are steeper than the $q_i$ and thus have to run through the strip $S$ between $q_0,q_1$ in $\R^2$ and after shifting the parameters and applying suitable translations, we have $\pi v_i$ in some compact set $K\subset\Int S$. With proposition \ref{N semi-cont} and corollary \ref{alpha C^1} we obtain a limit $v_+\in\NN_{\eta_+}$ with $\pi v_+\in K$. By (i) and the assumption that $q_0,q_1$ are neighbors, $c_+:=c_{v_+}$ has to be some heteroclinic and we have to show that it has the right asymptotic behavior. Put some $\tau$-periodic curve $\gamma$ into $\Int S$. Then the $c_{v_i}$ intersect $\gamma$ for the first time $t_i$ ''upwards'', i.e. $c(-\infty,t_i]$ is contained in the closed strip below $\gamma$. Now fix $K\subset\Int(S)$, s.th. $K$ contains a segment of $\gamma$ projecting surjectively to $p(\gamma)\subset\T$. Shift the times $t_i$ to $0$ and apply some translate $\tau$ to have $\tau \pi v_i\in K\cap \gamma$. We still get a heteroclinic limit, but now with the prescibed behavior.
\end{proof}

\subsection{The proof of theorem II}

Every $\eta$-semistatic $c:\R\to\T$ lifts to a minimal geodesic $\tilde c:\R\to\R^2$. Conversely we show here that every arc-length minimal geodesic $\tilde c:\R\to\R^2$ projects to an $\eta$-semistatic for some $\eta \in H_F$.

\begin{prop}\label{asymptotic directions min rays}
Let $c:[a_0,\infty)\to\R^2$ be a minimal geodesic ray. Then the asymptotic direction $\delta^+(c)\in S^1$ exists.
\end{prop}

\begin{bemerk}\label{bem asymptotic directions min rays}
In fact, there exists a global constant $D=D(F)\geq 0$, s.th. $c[a_0,\infty)$ has distance $\leq D$ from the straight line $c(a_0)+\R\delta^+(c)$. The proof in the Riemannian case is due to Hedlund (\cite{hedlund}, lemma 7.1), but it carries over to the Finsler case, cf. pp. 438f in \cite{zaustinsky}.
\end{bemerk}

\begin{proof}
In the proof of theorem \ref{existence rot vector}, we used only the minimality of $c$ w.r.t. $l_F$ in $\R^2$, so the proof carries over to minimal rays.
\end{proof}

Using $\delta^+(c)$, we can associate a homology class to minimal rays $c:[a_0,\infty)\to\T$ by setting
\[ \tilde\rho(c) := \R_{\geq 0} \delta^+(c) \cap G_F. \]
Note that $\tilde \rho(c)$ might a priori be different from $\rho^+(c)$, if $\rho^+(c)$ exists at all. However, we have the following result, showing that all statements about semistatic geodesics carry over to minimal geodesics.

\begin{prop}\label{minimals semistatic}
Let $F$ be Finsler metric on $\T$, $c:[a_0,\infty)\to \T$ be an arc-length minimal ray and $h:= \tilde\rho(c)\in G_F$. Then there exists $\eta\in\F^h$ and $u\prec L_\eta$, such that $c$ is calibrated on $[a_0,\infty)$ w.r.t. $u,\eta$. In particular, if $c$ is minimal on $\R$, then $\dot c\in \J(u) \subset \NN_\eta$. The analogous result holds for rays $c:(-\infty,a_0]\to\T$.
\end{prop}

We need the following basic observation, which shows that we can always approximate the Mather sets $\M^h$ ''from both sides''.

\begin{lemma}\label{mather sets approximierbar}
For any $h\in G_F, v\in\M^h$ there are sequences $h_n^\pm \in G_F$ with $h_n^-\nearrow h, h_n^+ \searrow h$ (in the cyclic order of $G_F$) and $v_n^\pm\in\M^{h_n^\pm}$ with $v_n^\pm\to v$.
\end{lemma}

\begin{proof}[Proof of \ref{mather sets approximierbar}]
First, let $h$ be rational, then $c_v$ is a periodic orbit with some period $T>0$. Taking any $h_n^\pm, v_n^\pm\in\M^{h_n^\pm}$, shifting the parameter of $c_{v_n^\pm}$ and applying a suitable translate $\tau_z$, s.th. $\pi(v_n^\pm)\in c_v[0,T]$ (this is possible since w.l.o.g. $h_n^\pm\notin \R h$), we obtain a convergent subsequence $v_n^\pm \to \dot c_v(t_0)\in\M^h$ for some $t_0\in[0,T]$ by $\pi^{-1}(x)\cap \NN_\eta=\pi^{-1}(x)\cap \M^h$ for any $\eta\in \F^h, x\in\pi(\M^h)$ and using the semi-continuity of $\NN_\eta$, cf. proposition \ref{N semi-cont}. The claim follows after applying $\phi^{-t_0}$.

Now let $h$ be irrational. The above arguments show that we obtain a limit in $\M^h$ from $\M^{h_n^\pm}$ (since the footpoints project into the compact set $\pi(\M^h)\subset \T$, if $\pi(v_n^\pm)\in c_v(\R)$ in $\R^2$). By the minimality of $\M^h$, $v$ itself is such a limit (use a diagonal argument).
\end{proof}

\begin{proof}[Proof of \ref{minimals semistatic}]
Write $c$ also for a lift $c:\R\to\R^2$.

\underline{Step 1} ($c$ cannot cross any geodesics in $\pi(\M^h)\subset\R^2$). Suppose $c(a_0)\in\pi(\M^h)$, but $\dot c(a_0) \neq v := \pi^{-1}(c(a_0))\cap\M^h$ and let $v_n^\pm\to v$ from lemma \ref{mather sets approximierbar}. From the asymptotic behavior of $c, c_{v_n^\pm}$ and $v\neq \dot c(a_0)$, we obtain successive intersections of either $c$ and $c_{v_n^-}$ or $c$ and $c_{v_n^+}$ for large $n$, contradicting the minimality of both curves w.r.t. $l_F$.

\underline{Step 2} (there is a geodesic $q_0:\R\to\R^2$ from $\M^h$ and $t_n\to\infty$, s.th. in $\R^2$ the distances $d(c(t_n),q_0(\R))$ tend to $0$). By proposition 2 in Mather's paper \cite{mather}, there is a geodesic $q_0$ from $\M^h$ and a sequence $t_n\to\infty$, s.th. in $\T$ we have $d(c(t_n),q_0(\R))\to 0$. But by step 1, $c$ is contained in a fixed gap of $\pi(\M^h)\subset \R^2$, so the claim follows (in fact, due to the special structure of $\M^h$ in the irrational case, we apply Mather's result only in the rational case).

\underline{Step 3} ($c$ is calibrated). In the rational case let e.g. $q_0 < q_1$ be neighboring minimizers, s.th. $c$ is contained in the strip between $q_0,q_1$ and $\eta=\eta_-\in \F^h=[\eta_-,\eta_+]$. Take some forward weak KAM solution $u\prec L_\eta$ having a curve $c_0$ between $q_0,q_1$ as a calibrated curve, such that $c_0$ is asymptotic in $+\infty$ to $q_0$ (cf. theorem \ref{rational directions} (iii)). In the irrational case, take any forward weak KAM solution $u\prec L_\eta$ for $\eta=\nabla\beta(h)$.

Now, if $c$ is not $(u,\eta)$-calibrated on $[a_0,\infty)$, there exists a $(u,\eta)$-calibrated curve $\gamma:[0,\infty)\to\T$ with $\gamma(0)=c(a)$ and $\dot\gamma(0)\neq \dot c(a)$ for some $a>a_0$. By our previous results we obtain $\lim_{t\to\infty} d(\gamma(t),c_0(\R)) = 0$, hence there is a sequence $s_n\to \infty$ with $d(c(t_n),\gamma(s_n))\to 0$. This contradicts the minimality of both curves w.r.t. $l_F$.
\end{proof}

\begin{proof}[Proof of theorem II]
Combine theorems \ref{irrational directions} and \ref{rational directions}, remark \ref{bem asymptotic directions min rays} and proposition \ref{minimals semistatic}. The only thing that is left open is the Lipschitz property of the two sets $\G_\delta\cup\G_\delta^\pm$, which we will prove in proposition \ref{M^pm} and remark \ref{bem M^pm}.
\end{proof}

%%%%%%%%%%%%%%%%%%%%%%%%%%%%%%%%%%%%%%%%%%%%%%%%%%%%%%%%%%%%%%%%%%%%%%%%%%%%%%%%%%%%%%%%%%
%%%%%%%%%%%%%%%%%%%%%%%%%%%%%%%%%%%%%%%%%%%%%%%%%%%%%%%%%%%%%%%%%%%%%%%%%%%%%%%%%%%%%%%%%%
%%%%%%%%%%%%%%%%%%%%%%%%%%%%%%%%%%%%%%%%%%%%%%%%%%%%%%%%%%%%%%%%%%%%%%%%%%%%%%%%%%%%%%%%%%
%%%%%%%%%%%%%%%%%%%%%%%% Section invariant tori in rational directions %%%%%%%%%%%%%%%%%%%
%%%%%%%%%%%%%%%%%%%%%%%%%%%%%%%%%%%%%%%%%%%%%%%%%%%%%%%%%%%%%%%%%%%%%%%%%%%%%%%%%%%%%%%%%%
%%%%%%%%%%%%%%%%%%%%%%%%%%%%%%%%%%%%%%%%%%%%%%%%%%%%%%%%%%%%%%%%%%%%%%%%%%%%%%%%%%%%%%%%%%
%%%%%%%%%%%%%%%%%%%%%%%%%%%%%%%%%%%%%%%%%%%%%%%%%%%%%%%%%%%%%%%%%%%%%%%%%%%%%%%%%%%%%%%%%%

\section{Topological entropy and invariant tori, proof of theorem I}\label{top ent and invar tori}

We fix again a Finsler metric $F$ on $\T$ with geodesic flow $\phi^t:S\T\to S\T$. Theorem I claims the existence of $\phi^t$-invariant graphs in $S\T$ in the case of $\h(\phi^t,S\T)=0$. We shall prove it in two steps: \begin{itemize}
\item[(1.)] Show that $\h(\phi^t,S\T)=0$ implies the existence of invariant graphs $\TT_h^0$ in $S\T$ for rotation vectors $h\in G_F$ with rational slope, such that for all $v\in\TT_h^0$ we have $\rho(c_v)=h$. 
\item[(2.)] Take limits of the $\TT_h^0$ to obtain invariant graphs for all $h\in G_F$.
\end{itemize}

In subsections \ref{section J}, \ref{multibump} and \ref{gap and entropy} we make the first step, the proof of the main theorem (step 2) is contained in subsection \ref{step 2}. The methods in subsections \ref{section J} and \ref{multibump} are similar to the techniques of Bosetto and Serra \cite{bosetto serra}, cf. also \cite{rabinowitz2}, while the mentioned articles work in the more restricted setting of non-autonomous Tonelli Lagrangian systems with one degree of freedom. Before Bosetto and Serra, Rabinowitz and Bolotin used similar techniques, cf. \cite{rabinowitz1}, \cite{rabinowitz}.

\subsection{The gap-condition}\label{section J}

The invariant tori $\TT_h^0$ for $h\in G_F$ with rational slope will consist of the periodic minimizers $\M^h$ together with heteroclinic orbits from $\NN_\eta, \eta\in\F^h$ between neighboring periodic minimizers (recall theorem \ref{rational directions}). We work in the universal cover $\R^2$ and introduce some notation.

\begin{defn}\label{def N^pm}
Let $h\in G_F$ have rational slope and let $q_0,q_1:\R\to\R^2$ be neighboring periodic minimizers from $\M^h$ with $q_0<q_1$ w.r.t. the ordering $<$ in definition \ref{ordering}. Write $S=S(q_0,q_1)\subset \R^2$ for the closed strip between $q_0,q_1$ and define
\begin{align*}
& \Om = \Om_h(q_0,q_1) := \{ c\in C_{loc}^{ac}(\R, S(q_0,q_1)) : \la h , c(t)\ra\to \pm\infty, t\to \pm\infty \}, \\
& c(\pm \infty)=q_i \quad : \iff \quad d(c(t), q_i(\R)) \to 0, ~ t\to \pm\infty, \\
& \Om^\pm = \Om_h^\pm(q_0,q_1) := \{ c\in \Om_h(q_0,q_1) : c(\mp\infty) = q_0, c(\pm\infty) = q_1 \} .
\end{align*}
We denote the set of heteroclinics between $q_0,q_1$ by
\[ \NN^\pm = \NN_h^\pm(q_0,q_1) := \left \{ v\in \bigcup_{\eta\in \F^h} \NN_\eta ~:~  c_v \in \Om_h^\pm(q_0,q_1) \right \}. \]
\end{defn}

We wish to build the $\TT_h^0$ from $\M^h$ together with one of the sets $\NN_h^\pm(q_0,q_1)$ between neighboring minimizers $q_0,q_1$. Of course, this is not always possible. Observe that
\[ \pi(\NN_h^\pm(q_0,q_1)) \subset \Int S(q_0,q_1). \]

\begin{defn}\label{def gap}
We say that $F$ fulfills the \emph{gap-condition}, if there are a rational $h\in G_F$ and neighbors $q_0<q_1$ from $\M^h$, such that
\[ \pi(\NN_h^-(q_0,q_1))\neq \Int S(q_0,q_1) \quad \text{and} \quad \pi(\NN_h^+(q_0,q_1))\neq \Int S(q_0,q_1). \]
If $F$ does not fulfill the gap-condition, denote by $\TT_h^0$ a closed invariant set in $S\T$ built by taking $\M^h$ and by choosing for any two neighboring minimizers $q_0,q_1$ from $\M^h$ one of the sets $\NN_h^\pm(q_0,q_1)$, that has $\pi(\NN_h^\pm(q_0,q_1)) = \Int S(q_0,q_1)$.
\end{defn}

\begin{bemerk}
\begin{itemize}
\item[(i)] It follows from the definition, that $\pi(\TT_h^0) = \R^2$.

\item[(ii)] We will show that the gap-condition implies positive topological entropy of the geodesic flow in $S\T$. Intuitively, if $q_0,q_1$ are hyperbolic neighboring periodic minimizers, the gap-condition corresponds to a transverse intersection of stable and unstable manifolds, cf. theorem 5.3 in \cite{bosetto serra}. Note that the intersection of these stable and unstable manifolds is never empty by the existence of heteroclinics between $q_0,q_1$. Also we mention the connection to Katok's result \cite{katok1}, stating that $\h(\phi^t,S\T)>0$ is equivalent to the existence of a horse shoe $\Lambda\subset S\T$.

\item[(iii)] If $c:\R\to\T$ is a contractible closed geodesic, then in $S\T$ there are no invariant graphs for the geodesic flow and the gap-condition holds: for such $c$ the curve of velocity vectors $\dot c:\R\to S\R^2\cong \R^2\times S^1$ of the lifted $c$ is closed, non-contractible and would break through any lifted invariant graph.
\end{itemize}
\end{bemerk}

{\bf Notation.} For subsections \ref{section J} and \ref{multibump} we fix a rational $h\in G_F$ and neighboring minimizers $q_0<q_1$ from $\M^h$ and suppress $h,q_0,q_1$ in the notation. Moreover let $\theta>0$ be the common prime period of the $q_i$ and $\tau= \tau_z$ the prime translation with $z\in \R_{> 0} h \cap \Z^2$. \vspace*{6pt}

In order to understand the gap-condition, one has to study the heteroclinics in $\NN^\pm$ in greater detail. The goal for the rest of this subsection is to define an asymptotic action-functional $J:\Om^\pm\to \R$ that characterizes the heteroclinic semistatics in $\Om^\pm$. This should be completely analogous to ascribing to a periodic curve $c:[0,T]\to\T$ its Lagrangian action $A_{L_\eta}(c)$, where $A_{L_\eta}(c)=0$ was equivalent to $c$ being minimal.

\begin{defn}\label{J-def}
Set
\begin{align*}
& A := A_{L_F-h_0+1/2} : C^{ac}([a,b],\R^2)\to \R, \quad  h_0 := h / \|h\|^2.
\end{align*}
For a curve $c\in \Om$ with $c(-\infty)=q_i,c(\infty)=q_j$ we say that a pair of sequences of real numbers $(s_n,t_n)$ is a \emph{$J$-sequence}, if
\[ s_n\to-\infty, ~ t_n\to \infty, \quad  d(c(s_n), q_i(-n\theta)), ~ d(c(t_n), q_j(n\theta))  ~ \to 0 \]
and define the \emph{asymptotic action $J(c)$} by
\[ J(c) := \lim_{n\to\infty} A(c[s_n,t_n]). \]
\end{defn}

We will show in a moment that $J$ is well-defined.

\begin{bemerk}\label{bem J-def}\begin{itemize}
\item[(i)] From $\la \eta, h \ra = 1$ for $\eta\in \F^h$ we have $\eta_0 := \eta-h_0 \perp h$, so the function $\hat\eta_0 := \la \eta_0, . \ra:\R^2\to\R$ is $\tau$-invariant and bounded in $S$.

\item[(ii)] Note that for curves $c\in\Om$ with known asymptotic behavior, there always exist $J$-sequences.

\item[(iii)] We have again $A(q_i[0,\theta])=0$ and for any $\tau^k$-periodic curve $c$ we have $A(c)\geq 0$ by the definition of \mane's critical value and $\tau$-invariance of $\hat\eta_0$. Moreover one easily shows, using the definition of \mane's critical value and the bound for $\hat\eta_0$ in $S$, that
\[ \exists B=B(F,h)\in\R : \qquad A(c[a,b]) \geq B \quad \forall c\in\Om, a\leq b. \]
From this we also have
\[ J(c) \in [B,\infty]. \]

\item[(iv)] A similar functional $J$ is considered in \cite{rabinowitz}. See also \cite{bosetto serra}.
\end{itemize}\end{bemerk}

We will frequently use the following. Recall $\frac{1}{c_F}\norm \leq F \leq c_F \norm$ by lemma \ref{c_F}.

\begin{defn}\label{def b}
For $x,y\in \R^2$ we write $\es_{x,y}(t)=x+t\frac{y-x}{d(x,y)},t\in [0,d(x,y)]$ for the euclidean straight segment from $x$ to $y$. Set
\[ b := \frac{c_F^2+1}{2} + \|h_0\| . \]
\end{defn}

\begin{bemerk}
The following calculation will be useful several times:
\begin{align*}
\left | A(\es_{x,y} ) \right | \leq \left | \int_0^{d(x,y)} \frac{1}{2}F^2(\dot\es_{x,y}) + \frac{1}{2} dt \right | + \|h_0\| \cdot d(x,y) \leq b \cdot d(x,y).
\end{align*}
\end{bemerk}

\begin{lemma}\label{J well-def}
$J(c)$ is well-defined, i.e. the limit exists and is independent of the choice of the $J$-sequence $(s_n,t_n)$.
\end{lemma}

\begin{bemerk}
It follows directly that $J$ is invariant under $\tau$ and time shifts, i.e. $J(c)=J(\tilde c)$ for $\tilde c(t)=\tau^k c(t+t_0)$.
\end{bemerk}

\begin{proof}
Let $(s_n, t_n),(\tilde s_n,\tilde t_n)$ be two $J$-sequences for $c\in\Om$ with asymptotic limits $c(-\infty)=q_i,c(\infty)=q_j$ and for $n\in \Z$ set
\[ \e_n := \begin{cases} d(q_i(n\theta),c(s_n)) & : n \leq 0 \\ d( q_j(n\theta),c(t_n)) & : n >0 \end{cases}, \quad \tilde\e_n := \begin{cases} d(q_i(n\theta),c(\tilde s_n)) & : n \leq 0 \\ d(q_j(n\theta),c(\tilde t_n)) & : n >0 \end{cases} , \]
so $\e_n,\tilde\e_n\to 0$ for $|n|\to\infty$. For $n,m\in \N$ with $n \leq m$ and $s_m \leq \tilde s_n$ we see by the $A$-minimality of $q_i$ that
\begin{align*}
0 & =  A(q_i[-m\theta, -n \theta]) \leq A(\es_{q_i(-m\theta),c(s_m)} * c|_{[s_m, \tilde s_n]} * \es_{c(\tilde s_n), q_i(-n \theta)}) \\
& \leq b (\e_{-m} + \tilde\e_{-n})  + A(c[s_m, \tilde s_n])
\end{align*}
and analogously for $[\tilde t_n,t_m]$, if $t_m \geq \tilde t_n$, i.e.
\begin{align*}
A(c[s_m, t_m]) - A(c[\tilde s_n, \tilde t_n]) \geq - b (\e_{-m} + \tilde\e_{-n} + \tilde\e_n + \e_m ) .
\end{align*}
If $(s_n,t_n)=(\tilde s_n,\tilde t_n)$, taking certain subsequences $m_k,n_k$ shows that
\[ \liminf A(c[s_n,t_n])=\limsup A(c[s_n,t_n]), \]
i.e. the limit $J(c)$ exists. If $\tilde J(c)$ is the value obtained from taking $(\tilde s_n,\tilde t_n)$ instead of $(s_n,t_n)$ we find by the above arguments that $J(c)\geq \tilde J(c)$. Analogously one shows $\tilde J(c)\geq J(c)$.
\end{proof}

An important property of $J$, as for the Lagrangian action, is semi-continuity.

\begin{prop}[lower semi-continuity of $J$]\label{J semi-cont}
Let $K\subset S$ be a compact set, $c^n\in\Om$ a sequence with $c^n(0)\in K$ for all $n$ and assume
\[ \liminf J(c^n) < \infty. \]
Then there exists a $C^0_{loc}$-convergent subsequence $c^{n_l}$ with a limit $c\in\Om$. Moreover, if all $c^{n_l}$ and $c$ have the same asymptotic behavior, then
\[ J(c) \leq \liminf J(c^{n_l}) . \]
\end{prop}

\begin{proof}
Let $D > \liminf J(c^n)$ and w.l.o.g. $c^n$, such that $J(c^n)\leq D$ for all $n$. Using remark \ref{bem J-def} (iii), we find for all $k\in\N$ that $A(c^n[-k,k])\leq D-2B$. Consider
\[ \Gamma_k := \{ c\in C^{ac}([-k,k], S) : c(0)\in K , A(c[-k,k]) \leq D-2B \}. \]
Estimating the length of $c[-k,k]$ for $c\in \Gamma_k$ using $A(c)\leq D-2B$ one shows that $c[-k,k]$ is contained in a compact set in $K_k\subset\R^2$. Hence, by proposition \ref{A semi-cont} the sets $\Gamma_k$ are compact w.r.t. $C^0$-convergence. Starting with $k=1$, we take convergent subsequences of $c^n|_{[-k,k]}$ and by a diagonal argument obtain a $C^0_{loc}$-convergent subsequence $c^{n_l}$ of $c^n$ with limit $c\in C^{ac}_{loc}(\R,S)$. Moreover, using $J(c)<\infty$, one easily shows that $\la c(t),h \ra \to \pm\infty$ for $t\to\pm\infty$, so $c\in\Om$.

Now assume w.lo.g. $c=\lim c^n$ and that all $c^n,c$ have the same asymptotic behavior, say $c(-\infty)=q_i,c(\infty)=q_j$. Pick $\e,\delta>0$, let $(s_l,t_l),(s_l^n,t_l^n)$ be $J$-sequences for $c,c^n$ and $l_0$ large, s.th. $c(s_{l_0}),c(t_{l_0})$ lie $\delta/2$-close the corresponding limit $q_i,q_j$. For large $n$, the $C^0$-convergence on $[s_{l_0},t_{l_0}]$ of $c^n\to c$ forces $c^n(s_{l_0}),c^n(t_{l_0})$ to lie $\delta$-close to $q_i,q_j$, respectively. Moreover, we can find large $l_n> l_0$ with $c^n(s_{l_n}^n), c^n(t_{l_n}^n)$ also lying $\delta$-close to its limit and $A(c^n[s_{l_n}^n,t_{l_n}^n]) \leq D+\e$. Arguing just as in the proof of lemma \ref{J well-def} we find
\[ D \geq A(c^n[s_{l_n}^n,t_{l_n}^n]) - \e \geq A(c^n[s_{l_0},t_{l_0}]) - 4b \delta - \e. \]
The semi-continuity of $A$ shows $A(c[s_{l_0},t_{l_0}])\leq D+4b\delta+\e$. $\e>0$ is arbitrary in this inequality (it was used to find the $l_n$) and with $l_0\to \infty$, we can take the limit $\delta\to 0$. Hence $J(c)\leq D$, while $D>\liminf J(c^n)$ was arbitrary. 
\end{proof}

We can now characterize the semistatics in $\Om^\pm$ using $J$.

\begin{prop}\label{M^pm}
Set
\[ \om^\pm := \inf\{ J(c) : c\in\Om^\pm \} \in [B,\infty). \]
Then
\[ \forall c\in\Om^\pm : \qquad \dot c\in \NN^\pm \quad \iff \quad J(c)= \om^\pm. \]
\end{prop}

\begin{proof}
We argue in the setting of $\Om^+$, the case $\Om^-$ is analogous.

Pick a curve $c_0\in\Om^+$ with $\dot c_0\in\NN^+$, some $\eta\in\F^h$ and a dominated function $u\prec L_\eta$ with $\dot c_0\in\J(u)$. With $\hat\eta_0=\la\eta-h_0,.\ra: \R^2\to\R$ as above set
\[ \hat u := u \circ p+\hat\eta_0 : \R^2\to \R, \]
where $p:\R^2\to\T$ is the covering map. Observe that for $c:[a,b]\to\R^2$ we have
\[ \hat u \circ c|_a^b = u \circ pc|_a^b + \hat\eta_0 \circ c|_a^b \leq A_{L_\eta}(pc[a,b]) + \hat\eta_0 \circ c|_a^b = A(c[a,b]) \]
with equality iff $p\circ c$ is $(u,\eta)$-calibrated on $[a,b]$. If $c\in\Om^+$ and $(s_n,t_n)$ is a $J$-sequence for $c$, we obtain using the $\tau$-invariance of $\hat\eta_0$ that for $a\leq b$
\begin{align*}
J(c) & = A(c[a,b]) + \lim \left ( A(c[s_n,a])+A(c[b,t_n]) \right ) \\
& \geq \hat u\circ c|_a^b + \lim \left (\hat u\circ c|_{s_n}^a + \hat u\circ c|_b^{t_n} \right ) = \hat u|_{q_0(0)}^{q_1(0)},
\end{align*}
where equality holds iff $\dot c\in\J(u)$, since $a,b$ were arbitrary. In particular $\om^+ = \hat u|_{q_0(0)}^{q_1(0)}$ by the existence of $c_0$. We proved the following: If $\dot c_0\in\NN^+$, then $J(c_0)=\om^+$. Conversely, we saw that if $J(c)=\om^+$, then $\dot c\in\J(u)\subset \NN^+$.
\end{proof}

\begin{bemerk}\label{bem M^pm}
In the proof we saw that
\[ \NN^\pm = \J(u)\cap \pi^{-1}(\Int S) \]
for any $\eta\in\F^h$ and $u\prec L_\eta$ having just one $c_0\in\Omega^\pm$ as a calibrated curve. In particular the sets $\NN^\pm=\NN_h^\pm(q_0,q_1)\subset S\T$ are contained in two Lipschitz graphs over $0_\T$, the Lipschitz constant depending only on $F$ and hence the closed invariant set $\TT_h^0$ in definition \ref{def gap} is an invariant Lipschitz graph.
\end{bemerk}

We can now concretize the gap-condition. It is equivalent to the case where for some rational $h\in G_F$ there is no invariant graph $\TT_h^0$ as defined in \ref{def gap}. In this case we know two things:
\begin{itemize}
\item[(i)] The periodic minimizers of direction $h$ do not foliate the torus $\T$, i.e. $\pi(\M^h)\neq \T$.
\item[(ii)] There are neighbors $q_0,q_1$ in $\M^h$, such that both $\NN^\pm=\NN_h^\pm(q_0,q_1)$ do not foliate the strip $\Int S = \Int S(q_0,q_1)$, i.e. $\pi(\NN^\pm) \neq \Int S$.
\end{itemize}

\begin{lemma} \label{q_i neighboring}
Let $F$ fulfill the gap-condition and $\delta > 0$ be sufficiently small, such that there are $x,y\in S$ with $d(x,\pi(\M^h))\geq \delta$ and $d(y,\pi(\NN^\pm))\geq \delta$. \begin{itemize}
\item[(i)] Set $\Om_1: = \bigcup_{T>0}\{ c\in C^{ac}([0,T],\R^2) : ~ c(T)=\tau c(0) \}$ and
\[ \om(\delta) := \inf \left \{ A(c) : ~ c\in \Om_1,~ c(0)\in S,~ d(c(0),\pi(\M^h))\geq \delta \right \} \]
Then $\om(\delta)> \om(0)=0$.

\item[(ii)] Set
\[ \om^\pm(\delta) := \inf \left \{ J(c) : ~ c\in \Om^\pm ,~ c(0)\in S, ~ d(c(0),\pi(\NN^\pm)) \geq \delta \right \} \]
Then $\om^\pm(\delta)>\om^\pm(0)=\om^\pm$.
\end{itemize}
\end{lemma}

\begin{proof}
(i) $\om(\delta)\geq 0$ by the remark \ref{bem J-def} (iii). Suppose $\om(\delta)=0$, so there are $T_n>0, c^n\in C^{ac}_{loc}([0,T_n], \R^2)$ with
\[ c^n(T_n)=\tau c^n(0)\in S, \quad d(c^n(0),\pi(\M^h))\geq \delta, \quad A(c^n)\to 0. \]
W.l.o.g. $F(\dot c^n)=1$ a.e. by proposition \ref{F=1 optimal}, s.th.
\[ T_n - \la h_0, z \ra = T_n - \la h_0, c^n \ra|_0^{T_n} = A(c^n)\to 0 \quad \Rightarrow \quad T_n \to \la h_0, z \ra = \theta. \]
Applying suitable $\tau^k$ we may assume that $c^n(0)\in K$ for some compact $K\subset S$ and changing the parametrisation slightly to obtain $T_n=\theta$ for all $n$, we still have $A(c^n)\to 0$ (one calculates the action of $c^n(\frac{T_n}{\theta} t)$ to be $\frac{1}{2\theta}(T_n^2-\theta^2)\to 0$). With proposition \ref{A semi-cont} we obtain a convergent subsequence of $c^n$ with a $\tau$-periodic limit $c\in C^{ac}([0,\theta],\R^2)$ and $A(c)=0$ (observe that $l_F(c^n)=T_n$ is bounded, so the $c^n[0,\theta]$ lie in a fixed compact set). This shows $\dot c\in\M^h$ ($c$ is of minimal length $\theta$ in the homotopy class $z$), contradicting $d(c(0),\pi(\M^h))\geq \delta>0$.

(ii) By definition we have $\om^\pm(\delta)\geq \om^\pm$. Suppose $\om^+(\delta)=\om^+$. Then there are $c^n\in \Omega^+$ with $J(c^n)\to \om^+$ and after applying suitable $\tau^k$ we can assume $c^n(0)\to x\in S$ with $d(x,\pi(\NN^\pm))\geq \delta$ (in particular $x\notin\partial S$). Set
\[ S^+ := \Con_x (\Int (S) - \pi(\NN^\pm)). \]
$S^+$ is bounded by two curves $c_0,c_1\in\Om^+$ with $J(c_i)=\om^+$ and for large $n$, the points $c^n(0)$ also lie in $S^+$. If $c^n(\R)\not\subset \overline{S^+}$, it intersects one of the $c_i$. In this case replace $c^n$ by $c_i$ before/after the intersection, such that $c^n(\R)\subset \overline{S^+}$ for all $n$. Using that $c^n,c_i$ are asymptotic and $c_i$ is minimal w.r.t. $A$, it is easy to see that this does not increase $J(c^n)$ and hence w.l.o.g. $c^n(\R)\subset \overline{S^+}$ for all $n$. Now apply the lower semi-continuity of $J$ in proposition \ref{J semi-cont}. We obtain a limit $c=\lim c^n$ with $J(c)\leq \liminf J(c^n)=\om^+$. By proposition \ref{M^pm}, we have $\dot c(0)\in \NN^+$, contradicting $c(0)=x$ and $d(x,\pi(\NN^\pm))\geq\delta>0$.
\end{proof}

%%%%%%%%%%%%%%%%%%%%%%%%%%%%%%%%%%%%%%%%%%%%%%%%%%%%%%%%%%%%%%%%%%%%%%%%%%%%%%%%%%%%%%%%%%%%%%
%%%%%%%%%%%%%%%%%%%%%%%%%%%%%%%%%%%%%%%%%%%%%%%%%%%%%%%%%%%%%%%%%%%%%%%%%%%%%%%%%%%%%%%%%%%%%%
%%%%%%%%%%%%%%%%%%%%%%%%%%%%%%%%%%%%%%%%%%%%%%%%%%%%%%%%%%%%%%%%%%%%%%%%%%%%%%%%%%%%%%%%%%%%%%
%%%%%%%%%%%%%%%%%%%%%%%%%%%%%%%%%%%% multibump solutions %%%%%%%%%%%%%%%%%%%%%%%%%%%%%%%%%%%%%
%%%%%%%%%%%%%%%%%%%%%%%%%%%%%%%%%%%%%%%%%%%%%%%%%%%%%%%%%%%%%%%%%%%%%%%%%%%%%%%%%%%%%%%%%%%%%%
%%%%%%%%%%%%%%%%%%%%%%%%%%%%%%%%%%%%%%%%%%%%%%%%%%%%%%%%%%%%%%%%%%%%%%%%%%%%%%%%%%%%%%%%%%%%%%
%%%%%%%%%%%%%%%%%%%%%%%%%%%%%%%%%%%%%%%%%%%%%%%%%%%%%%%%%%%%%%%%%%%%%%%%%%%%%%%%%%%%%%%%%%%%%%

\subsection{Multibump solutions, if the gap-condition holds}\label{multibump}

We continue to work in the setting of subsection \ref{section J} and fix the triple $h,q_0,q_1$ for the rest of this section. Moreover we make the following
\abs

{\bf Assumption.} We assume for the rest of this section that the gap-condition from definition \ref{def gap} holds for the triple $h,q_0,q_1$.\abs

In this subsection we ask for dynamical consequences of the gap-condition, i.e. if there is no invariant torus $\TT_h^0$. We define ''switches'' $\s^\pm$ and using these switches we can prescribe oscillatory behavior. Write
\[ [i] := \begin{cases} 0 & : i \text{ even} \\ 1 & : i \text{ odd}\end{cases}  \qquad  \ep_i := \begin{cases} + & : i \text{ even} \\ - & : i \text{ odd}\end{cases}, \quad i\in\Z. \]

\begin{defn}[switches]\label{def switches}
Choose altogether four points
\[x_0^\pm,x_1^\pm \in \Int S - \pi(\NN^\pm), \]
such that $x_0^\pm,x_1^\pm$ lie in different connected components of $\Int (S)-\pi(\NN^\pm)$, and define the open sets
\[ S_i^\pm := \Con_{x_i^\pm} \left ( \Int (S) - \pi(\NN^\pm) \right ) \quad \subset S\subset\R^2 \]
(i.e. $\partial S_i^\pm$ consists of two semistatics from $\NN^\pm$ and no semistatic from $\NN^\pm$ runs into $\Int S_i^\pm$). Assume that the $S_0^\pm$ lie further left than the $S_1^\pm$ w.r.t. the orientation given by $h$. Choose $\delta>0$ and define the compact sets
\[ S_i^\pm(\delta) := \Clos\left ( \Con_{x_i^\pm} \left \{ x\in S_i^\pm: d(x, \partial S_i^\pm) > \delta \right \} \right ). \]
Choose curve segments $\gamma_i^\pm$ connecting $x_i^+$ to $q_{[i+1]}$, $x_i^-$ to $q_{[i]}$, respectively, and closed tubular neighborhoods $T_i^\pm$ of $\gamma_i^\pm$, such that
\begin{align*}
d(T_0^\pm, \text{ right part of } \partial S_0^\pm) > \delta ,\quad d(T_1^\pm, \text{ left part of } \partial S_1^\pm) > \delta.
\end{align*}
The \emph{switches} are the compact sets
\[ \s^-= \s^-(\delta) = \cup_{i=0,1} (S_i^-(\delta) \cup T_i^-) , \quad \s^+= \s^+(\delta) = \cup_{i=0,1} (S_i^+(\delta) \cup T_i^+). \]
Additionaly we choose two \emph{test curves} $c^\pm$ from $\NN^\pm$ with $c^\pm(\R)\cap \s^\pm=\emptyset$ and a straight euclidean segment $\sig_0$ orthogonal to $h$ connecting $q_0$ to $q_1$ having its interior in $\Int(S)$. Also choose $\kappa\in\N$, s.th.
\[ \la h , x \ra \leq \la h , y \ra \leq \la h , z \ra \quad \forall x\in \tau^{-\kappa}\sig_0, y \in\s^\pm, z \in \tau^{\kappa}\sig_0. \]
\end{defn}

\begin{figure}
\centering
\includegraphics[scale=1.0]{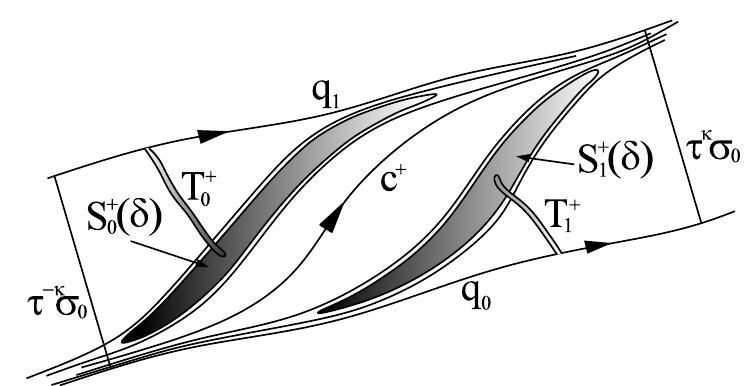}
\caption{The switch $\s^+$ with test curve $c^+$ and $\tau^{\pm\kappa}\sig_0$ constructed in definition \protect\ref{def switches}.}
\end{figure}

\begin{defn}[oscillating behavior]\label{def oscillating}
Pick some $\nu\in\N$ and a biinfinite sequence of integers $W = \{w_i\}_{i\in\Z}\subset \Z$ with $w_{i+1}\geq w_i + 2\kappa+\nu$. For $j\leq k$ write $\U^{jk}\subset S$ for the open set consisting of $\Int S - \cup_{j\leq i\leq k} \tau^{w_i}\s^{\ep_i}$ minus the regions left of $\tau^{w_{j-1}}c^{\ep_j}$ and right of $\tau^{w_{k+1}}c^{\ep_k}$ (w.r.t. the orientation given by $h$). 

Set
\begin{align*}
& \Omega^{jk} := \left \{ c\in\Omega ~:~ c(\R) \subset \overline{\U^{jk}} \right \}, \\
& \omega^{jk} := \inf\{ J(c) ~:~ c\in \Omega^{jk} \} .
\end{align*}
\end{defn}

\begin{bemerk}\label{bem def oscillating}\begin{itemize}
\item[(i)] We have $\omega^{jk} \in[B,\infty)$, where $B$ was defined in remark \ref{bem J-def}. Moreover $\om^{ii}=\om^{\ep_i}$ for all $i\in\Z$ in the notation of proposition \ref{M^pm} by the existence of the test curves.

\item[(ii)] By definition of $\Om^{jk}$, the curves $c\in\Om^{jk}$ satisfy
\[ c(-\infty) = q_{[j]}, \quad c(\infty) = q_{[k+1]}. \]

\item[(iii)] Note that if $c\in \Om$ is minimal for $J$ in an open set $U\subset \Int S$, each segment $c|_{[a,b]}$ is minimal for the action $A$ in $U$ and by proposition \ref{F=1 optimal} (iii), it is locally minimal for the Finsler-length and parametrised by arc-length. Hence $J$-minimal curves in open sets are arc-length geodesics and our goal is to find $J$-minimal curves $c\in \Om$ with $c(\R)\subset \U^{jk}$ and $J(c)=\om^{jk}$.
\end{itemize}\end{bemerk}

\begin{figure}
\centering
\includegraphics[scale=1.3]{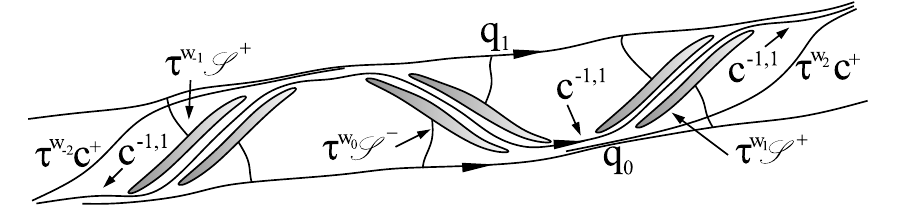}
\caption{A curve $c(\R)\subset \overline{\U^{jk}}$ is forced to have prescribed oscillating and asymptotic behavior. Here we see $\U^{-1,1}$ and the curve $c^{-1,1}$.}
\end{figure}

Another simple application of the semi-continuity of $J$ shows that $\om^{jk}$ is in fact a minimum. Later we will have to choose the right $\delta,\kappa, \nu$ in definitions \ref{def switches} and \ref{def oscillating} to show that minimizers in $\Om^{jk}$ are geodesics, i.e. $c(\R)\subset \U^{jk}$.

\begin{cor}\label{existence c^jk}
There exists $c^{jk}\in\Omega^{jk}$ with $J(c^{jk})=\omega^{jk}$ and $F(\dot c^{jk})=1$ a.e..
\end{cor}

\begin{proof}
Take a sequence $c^n\in\Om^{jk}$ with $J(c^n)\leq \om^{jk}+1/n$ and fix $c^n(0)$ in a compact set. Applying proposition \ref{J semi-cont} gives a limit curve $c=: c^{jk} \in\Om^{jk}$ (since $\overline{\U^{jk}}$ is closed). By definition of $\U^{jk}$, all $c^n,c$ have the same asymptotic behavior, showing $J(c)\leq \om^{jk}$ by proposition \ref{J semi-cont}. $F(\dot c)=1$ follows from the $A$-minimality of $c$ and proposition \ref{F=1 optimal}.
\end{proof}

Using the minimality of $c^{jk}$ we have the following.

\begin{prop}\label{linear oscillation}
Let $c^{jk}$ as in corollary \ref{existence c^jk}. \begin{itemize}
\item[(a)] There are $C_0,C_1>0$ (depending only on $\sig_0, \theta, \kappa, \nu$), s.th. for $t_0,T\in\R$
\[ c^{jk}(t_0)\in \tau^{w_j-\kappa}\sig_0, ~~ c^{jk}(t_0+T) \in \tau^{w_k+\kappa}\sig_0  \quad \Rightarrow \quad T \leq C_0+C_1 \cdot (w_k-w_j) .\]

\item[(b)] If $c^{jk}$ is disjoint from the tubes $\tau^{w_j}T_i^\pm$ in the translated switches, we have
\[ c^{jk}(\R)\cap \tau^{w_{j-1}}c^{\ep_j}(\R) = c^{jk}(\R)\cap \tau^{w_{k+1}}c^{\ep_k} (\R) = \emptyset, \quad c^{jk}(\R)\subset \Int(S). \]
\end{itemize}
\end{prop}

\begin{proof}
Let $c=c^{jk}$. (a) follows simply by comparing $c$ to the test curves between the switches, using $l_F$-minimality of $c$ and the assumption $w_{i+1}-w_i \geq \nu+2\kappa$.

(b) Let $c_0$ be any of the four curves $q_0,q_1,\tau^{w_{j-1}}c^{\ep_j},\tau^{w_{k+1}}c^{\ep_k}$. Suppose $c(I)\subset c_0(\R)$, where $I\subset\R$ is a maximal closed interval with this property. By definition of $\U^{jk}$, $I$ can only be of the form $[a,b],[a,\infty),(-\infty,b]$. We treat the first case, the others being analogous. Let $c_0(t_0)=c(b)$. By minimality of $c$, it has no selfintersections and hence $c(-\infty,b)$ lies in the connected component of $\R^2-\left (c[b,\infty)\cup c_0[t_0,\infty) \right )$ containing $c_0(t_0-\e,t_0)$ ($c(\R)$ is contained in the closure of one component of $\R^2-c_0(\R)$). This shows that $c(b-\e,b)$ is also a geodesic for $\e>0$ small: either this segment is part of $c_0(-\infty, t_0]$ or it lies in one connected component of $\R^2-c_0(\R)$ (an open set), where $c$ is disjoint from the $\tau^{w_j}T_i^\pm$ by assumption and hence geodesic by remark \ref{bem def oscillating} (iii). Now the two segments $c(b-\e,b),c(b,b+\e)$ are geodesisc, but $\dot c(b-)\neq\dot c(b+)$ (uniqueness of geodesics) and we can shorten $c$ at the vertex in $c(b)$ in such a way that the new curve is disjoint from $c_0$ near the vertex. This contradicts $c$'s the $l_F$-minimality.
\end{proof}

Choosing the right $\delta,\kappa,\nu$ in definitions \ref{def switches} and \ref{def oscillating}, the $c^{jk}$ are in fact geodesics, as we shall prove now. Intuitively, the $J$-minimizing curve $c^{jk}\in\Om^{jk}$ cannot intersect the sets $S_i^\pm(\delta)$ in the switches since, by lemma \ref{q_i neighboring}, the asymptotic action $J$ immediately increases. We could view the sets $S_i^\pm(\delta)$ as ''hilly'' areas in the geometrical landscape $(\R^2,F)$ and local minimizers $c\in\Om$ travel trough the valleys in this landscape, accompanying the test curves $c^\pm$.

\begin{thm}\label{free of constraints}
There exist $\delta > 0$ and $\kappa,\nu \in \N$, such that for $c^{jk}$ from corollary \ref{existence c^jk} we have
\[ c^{jk}(\R) \subset \U^{jk}. \]
In particular the $c^{jk}$ are locally minimizing arc-length geodesics.
\end{thm}

\begin{figure}
\centering
\includegraphics[scale=1.0]{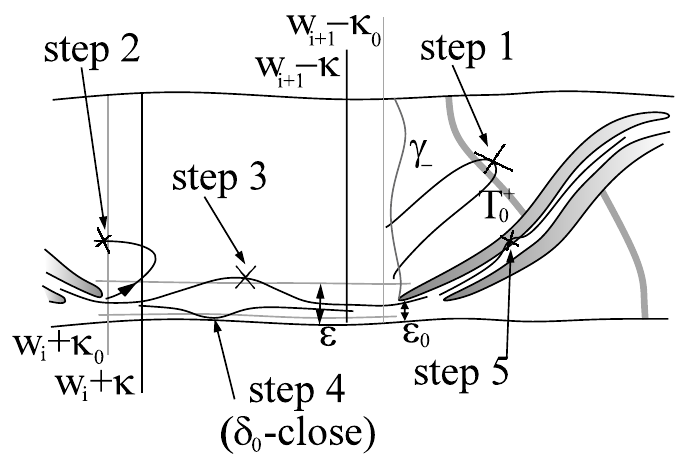}
\caption{The five steps in the proof of theorem \protect\ref{free of constraints}.}
\end{figure}

\begin{proof}
Choose the following real numbers:\begin{itemize}

\item[(1)] $\e$ with
\[ 0 < \e < \frac{1}{2} \inf_{s,t\in\R} d(q_0(s),q_1(t)) . \]

\item[(2)] $\e_0$ with
\[0< \e_0  < \frac{\om(\e)}{4b}. \]

\item[(3)] Some small $\delta>0$. Now construct the switches $\s^\pm,\sig_0,\kappa$ and the test curves $c^\pm$ in definition \ref{def switches} (where we assume that $\delta$ is sufficiently small in order to have $S_i^\pm(\delta)\neq\emptyset$). Make $\delta$ smaller, s.th. the following conditions are satisfied.

Set
\begin{align*}
 L:= \left \{\la h,.\ra = \min_{S_0^+(\delta)} \la h,.\ra \right \}\cap S , \quad R:= \left \{\la h,.\ra = \max_{S_1^+(\delta)} \la h,.\ra \right \}\cap S.
\end{align*}
\begin{itemize}
\item[(a)] In the connected component of $q_0$ in the strip between $L,\tau L$ in $S-S^+_0(\delta)$, any point $x$ has $d(x,q_0(\R))\leq \e_0$. Analogously, in the connected component of $q_1$ in the strip between $\tau^{-1}R,R$ in $S-S^+_1(\delta)$, any point $x$ has $d(x,q_1(\R))\leq \e_0$.

\item[(b)] Parametrise $\Con_{S_0^+(\delta)\cap L}L,\Con_{S_1^+(\delta)\cap R}R$ by curves $\gamma_L,\gamma_R:[0,1]\to S$ and let $C$ be an upper bound for the action of $\gamma_L^{\pm 1},\gamma_R^{\pm 1}$ on subintervals $[c,d]\subset[0,1]$ (where $\gamma_{L,R}^{-1}(t)=\gamma_{L,R}(1-t)$). Recall $T_0^+$ being the tube in definition \ref{def switches}. We impose that
\begin{align*}
& C + \|h_0\| \max \{\diam_{euc}(L),\diam_{euc}(R)\} \\
& \qquad \qquad \qquad \leq \inf \{ d_F(x,y) : x\in L\cup R, y\in T_0^+ \} .
 \end{align*}
\end{itemize}
That (a) and (b) can be met follows since $S_i^+(\delta)$ become longer and in the ends approach $q_0,q_1$, as $\delta\searrow 0$. We ask analogous conditions of $\s^-$.

\item[(4)] $\delta_0$ with
\[0< \delta_0  < \min_{i=0,1} \frac{\om^{\ep_i}(\delta)-\om^{\ep_i}}{4b+1} \]

\item[(5)] Adjust the choice of $\kappa\in\N$ from choice (3), s.th. the following conditions hold. \begin{itemize}
\item[(a)] Write $s_i^\pm$ for times where the test curves $c^{\ep_i}$ pass $\tau^{\pm\kappa}\sig_0$, respectively for $i=0,1$. Choose $\kappa$ so large that for any $a \leq s_i^-, s_i^+ \leq b$ the points $c^{\ep_i}(a),c^{\ep_i}(b)$ lie $\delta_0$-close to the corresponding limit $c^{\ep_i}(-\infty)=q_{[i]}, c^{\ep_i}(\infty)=q_{[i+1]}$ and for times $s,t\in \R$ satisfying
\[ d(c^{\ep_i}(a),q_{[i]}(s))\leq \delta_0 , \quad d(c^{\ep_i}(b), q_{[i+1]}(t))\leq \delta_0 \]
we have
\[ A(q_{[i]}[s^{\theta-}, s]) + A(c^{\ep_i}[a, b]) + A(q_{[i+1]}[t, t^{\theta+} ]) \leq \om^{\ep_i} + \delta_0 .\]
Here $s^{\theta-} := \max\{r\in\theta\Z: r\leq s\}$ and $t^{\theta+} := \min\{r\in \theta\Z: r\geq t\}$. That this is possible follows since $c^{\ep_i}$ converges to $q_{[i]},q_{[i+1]}$ in $C^1$ by the graph property of $\NN^\pm$ and from the convergence $A(c^{\ep_i}[s_n,t_n])\to \om^{\ep_i}$.
\item[(b)] The switches $\s^\pm$ are contained in the region between $\tau^{\pm\kappa_0}\sig_0$ for some $\kappa_0\in\N$. We ask
\[ \kappa > \frac{c_F l_F(\sig_0)}{2\|z\|} + \kappa_0,  \]
where $\tau=x\mapsto x+z$.
\end{itemize}

\item[(6)] $\nu\in\N$ with
\[  \nu > \frac{3b\e}{\om(\delta_0)}. \]
\end{itemize}

Now let $c=c^{jk}$ from corollary \ref{existence c^jk}. We prove the theorem in five steps.

\underline{Step 1.} ($c$ is disjoint from the tubes $\tau^{w_i}T_j^\pm$) Suppose e.g. $i$ is even and $c(t^*)\in \tau^{w_i}T_0^+$, then we find a compact segment $\tau^{-w_i}c[a,b]$ having endpoints, say, in $L$ from choice (3 b) and with $t^*\in[a,b]$. Then for $c,d\in[0,1]$ with $\gamma_L(c)=c(a), \gamma_L(d)=c(b)$ (assuming $c\leq d$, using $\gamma_L^{-1}$ in the other case) we have by minimality of $c$ w.r.t. the action $A$ that
\begin{align*}
C & \geq A(\gamma_L[c,d]) \geq A(c[a,b]) = l_F(c[a,b]) - \la h_0, c\ra |_a^b \\
& \geq l_F(c[a,b]) - \|h_0\| \|c(b)-c(a) \| \\
& \geq \inf \{ d_F(x,y) : x\in L, y\in T_0^+ \} - \|h_0\| \diam_{euc}(L),
\end{align*}
contradicting choice (3 b).

\underline{Step 2.} ($c$ traverses the region between $\tau^{\kappa_0}\sig_0,\tau^{\kappa}\sig_0$ only once) Suppose $c[a,b]$ has endpoints on $\tau^{w_i+\kappa_0}\sig_0$ and some $c(t^*)\in \tau^{w_i+\kappa}\sig_0$ for $t^*\in[a,b]$. Then by minimality of $c$
\begin{align*}
l_F(\sig_0) \geq A(c[a,b]) = l_F(c[a,b]) - \underset{=0}{\underbrace{\la h_0, c\ra |_a^b}} \geq \frac{1}{c_F} l_{euc}(c[a,b]) \geq \frac{2(\kappa-\kappa_0)\|z\|}{c_F}
\end{align*}
contradicting choice (5 b). So once $c$ passes $\tau^{w_i+\kappa}\sig_0$, it can not return to $\tau^{w_i+\kappa_0}\sig_0$. The same argument shows: Once $c$ passes $\tau^{w_{i+1}-\kappa_0}\sig_0$, it cannot return to $\tau^{w_{i+1}-\kappa}\sig_0$.

\underline{Step 3.} ($\e$-close to the $q_i$ between the switches) Let e.g. $i$ be even, so $c$ lies ''near'' $q_1$ between $\tau^{w_i}\s^+$ and $\tau^{w_{i+1}}\s^-$. Choose a time $t_0$ where $c(t_0)$ lies between $\tau^{w_i}\tau^{-1}R,\tau^{w_i}R$ at the right end of $\tau^{w_i}S^+_1(\delta)$ and $t_1$ where $c(t_1)$ lies between $\tau^{w_{i+1}}L,\tau^{w_{i+1}}\tau L$ at the left end of $\tau^{w_{i+1}}S^-_0(\delta)$ in choice (3 a). We can choose $t_0,t_1$ s.th. $c(t_0), c(t_1)$ are $\e_0$-close to points $x,y$ on $q_1$ congruent under some $\tau^m$. Let
\[ \gamma : = \es_{c(t_0), x} ~*~ q_1 ~*~ \es_{y,c(t_1)} . \]
Then by minimality of $c$ we find $2b \e_0 \geq A(\gamma) \geq A(c[t_0,t_1])$. By the assumption on $t_0,t_1$ we have
\[ d(c(t_1),\tau^mc(t_0)) \leq d(c(t_1),y) + d(\tau^mx,\tau^mc(t_0)) \leq 2\e_0. \]
Set $\gamma_0:= c|_{[t_0,t_1]}* \es_{c(t_1), \tau^m c(t_0)}$ and with lemma \ref{hedlund lemma} cut $\gamma_0$ into $m$ prime-periodic segments $\sig_1,...,\sig_m$. If $c|_{[t_0,t_1]}$ leaves the area $\e$-close to $q_1$, then so does $\gamma$ and we find $A(\sig_i)\geq \om(\e)$ for some $i$, while for the others $A(\sig_j)\geq 0$. This shows
\[ 2b \e_0 \geq A(c[t_0,t_1]) = A(\gamma_0) - A(\es_{c(t_1), \tau^m c(t_0)}) \geq \om(\e) - 2b\e_0, \]
contradicting choice (2).

\underline{Step 4.} (at some point $\delta_0$-close to the $q_i$) Let $t_i^\pm$ be the times with $c(t_i^\pm) \in \tau^{w_i\pm\kappa} \sig_0$, s.th. $c[t_i^+,t_{i+1}^-]$ lies entirely between $\tau^{w_i+\kappa}\sig_0$ and $\tau^{w_{i+1}-\kappa}\sig_0$. In $\tau^{w_{i+1}-\kappa}\sig_0$, the points $\tau^mc(t_i^+)$ and $c(t_{i+1}^-)$ are $\e$-close by step 3 (recall $\sig_0$ being a straight segment), where $m:=w_{i+1}-w_i-2\kappa \geq \nu$ (assumption of $W=\{w_i\}$). Arguing as in step 3, we close $c$ to a periodic curve $\gamma_0$ and assuming $c[t_i^+,t_{i+1}^-]$ is $\delta_0$-far away from $q_1$ everywhere, so is $\gamma_0$. Again, $\gamma_0$ is $\tau^m$-periodic but this time all the $\sig_i$ have $A(\sig_i)\geq \om(\delta_0)$, so as in step 3 by comparison of $c$ to $q_1$, we have
\[ \nu \om(\delta_0) \leq A(\gamma_0) \leq A(c[t_i^+,t_{i+1}^-]) + b \e \leq 3b\e,\]
contradicting choice (6).

\underline{Step 5.} ($c(\R)\subset \U^{jk}$) Choose $i \in \{j,...,k\}$ and write $t_-\in [t_{i-1}^+,t_i^-]$ and $t_+\in [t_i^+, t_{i+1}^-]$ for the times found in step 4, where
\[ d(c(t_-),q_{[i]}(T_-))\leq \delta_0 , \quad d(c(t_+),q_{[i+1]}(T_+))\leq \delta_0 \]
for some $T_-,T_+\in\R$ (for $i=j-1,k$ choose $|t_-|,|t_+|$, respectively, large such that the $\delta_0$-condition holds). By step 2 the segment $c[t_-,t_+]$ is the only part of $c$ near the switch $\tau^{w_i}\s^{\ep_i}$, so using step 1 we have to show that $c[t_-,t_+]$ and $\tau^{w_i}S_l^{\ep_i}(\delta)$ are disjoint for $l=0,1$. Write $a,b$ for times where the test curve $c^0 := \tau^{w_i}c^{\ep_i}$ passes $\delta_0$-close $q_{[i]}(T_-), q_{[i+1]}(T_+)$, respectively (cf. choice (5 a)), and set
\[ \hat c := q_{[i]}|_{(-\infty, T_-]} * \es_{q_{[i]}(T_-),c(t_-)} * c|_{[t_-,t_+]} * \es_{c(t_+), q_{[i+1]}(T_+)} * q_{[i+1]}|_{[T_+,\infty)}. \]
We compare $\hat c$ to $c^0$ and make use of $\om^\pm(\delta)>\om^\pm$ in lemma \ref{q_i neighboring}. If $c[t_-, t_+]$ runs into $\tau^{w_i}S^{\ep_i}_l(\delta)$, so does $\hat c\in\Omega^{\ep_i}$, i.e. $J(\hat c) \geq \om^{\ep_i}(\delta)$. Then with choice (5 a) and the minimality of $c$ we obtain
\begin{align*}
\om^{\ep_i}(\delta) & \leq J(\hat c) \leq A(c[t_-, t_+]) + A(q_{[i]}[T_-^{\theta-}, T_-]) + A(q_{[i+1]}[T_+, T_+^{\theta+}]) + 2b\delta_0 \\
& \leq A(c^0[a, b]) + A(q_{[i]}[T_-^{\theta-}, T_-]) + A(q_{[i+1]}[T_+, T_+^{\theta+}]) + 4b\delta_0  \\
& \leq \om^{\ep_i} + \delta_0 + 4b\delta_0
\end{align*}
This is a contradiction to choice (4).
\end{proof}

\subsection{The gap-condition implies positive entropy}\label{gap and entropy}

\begin{prop}\label{h_top=0}
Let $F$ be a Finsler metric on $\T$ that fulfills the gap-condition. Then the topological entropy of the geodesic flow $\phi^t$ of $F$ is positive, $\h(\phi^t,S\T)>0$.
\end{prop}

\begin{proof}
In $S\T$ we can take the euclidean product metric from $T\T \cong \T\times \R^2$. In particular, if base curves $c_v$ are seperated, so are the orbits $\phi^tv$ in $S\T$. The geodesics $c^{jk}$ from theorem \ref{free of constraints} oscillate on a length in $h$-direction bounded by $C\cdot (w_k-w_j)$, cf. proposition \ref{linear oscillation}. Choosing different sequences $W$ in definition \ref{def oscillating}, we optain an exponentially growing number of geodesics, that are $\e_1$-separated for $\e_1 := \min\{\e_0, \inf_{s,t\in\R}d(q_0(s),q_1(t))-2\e \}>0$ (by step 3 and choice (1) in the proof of theorem \ref{free of constraints}) in linearly bounded time in $\R^2$. But they are also $\e_1$-separated in $\T$: each $c^{jk}$ has to pass the left part of the switch $\tau^{w_j}\s^{\ep_j}$ and hence all $c^{jk}$ lie $\e_0$-close at time $t=0$, say, by choice (3 a) in the proof of theorem \ref{free of constraints}. Making $\e_0$ small, this shows that the curves have to separate in $\T$ for otherwise they would lift to curves that are $\e_1$-close. This shows
\[ 0< \limsup_{T\to\infty} \frac{\log s(T,\e_1)}{T} \leq \h(\phi^t,S\T) . \]
\end{proof}

%%%%%%%%%%%%%%%%%%%%%%%%%%%%%%%%%%%%%%%%%%%%%%%%%%%%%%%%%%%%%%%%%%%%%%%%%%%%%%%%%%%%%%%%%%%%%%
%%%%%%%%%%%%%%%%%%%%%%%%%%%%%%%%%%%%%%%%%%%%%%%%%%%%%%%%%%%%%%%%%%%%%%%%%%%%%%%%%%%%%%%%%%%%%%
%%%%%%%%%%%%%%%%%%%%%%%%%%%%%%%%%%%%%%%%%%%%%%%%%%%%%%%%%%%%%%%%%%%%%%%%%%%%%%%%%%%%%%%%%%%%%%
%%%%%%%%%%%%%%%%%%%%%%%%%%%%%%%%%%%% step 2 %%%%%%%%%%%%%%%%%%%%%%%%%%%%%%%%%%%%%%%%%%%%%%%%%%
%%%%%%%%%%%%%%%%%%%%%%%%%%%%%%%%%%%%%%%%%%%%%%%%%%%%%%%%%%%%%%%%%%%%%%%%%%%%%%%%%%%%%%%%%%%%%%
%%%%%%%%%%%%%%%%%%%%%%%%%%%%%%%%%%%%%%%%%%%%%%%%%%%%%%%%%%%%%%%%%%%%%%%%%%%%%%%%%%%%%%%%%%%%%%
%%%%%%%%%%%%%%%%%%%%%%%%%%%%%%%%%%%%%%%%%%%%%%%%%%%%%%%%%%%%%%%%%%%%%%%%%%%%%%%%%%%%%%%%%%%%%%

\subsection{Invariant tori for all rotation vectors}\label{step 2}

In this section we study Finsler metrics $F$ not fulfilling the gap-condition and prove theorem I from the introduction. We saw that in this case there are invariant Lipschitz graphs $\TT_h^0\subset S\T$ for all rational $h\in G_F$, cf. definition \ref{def gap} and remark \ref{bem M^pm}. Moreover, the Lipschitz constant of $\TT_h^0$ depends only on $F$.

We define the candidates for invariant tori found in section \ref{section mane sets}. Recall the notation $\NN_h^\pm(q_0,q_1)$ in definition \ref{def N^pm}.

\begin{defn}\label{def invar tori}
For $h\in G_F$ irrational write $\TT_h := \NN_{\nabla\beta(h)}$. For $h\in G_F$ rational write
\[ \TT_h^\pm := \M^h\cup \bigcup_{(q_0,q_1)\text{ neighboring in $\M^h$}} \NN_h^\pm(q_0,q_1) .\]
\end{defn}

\begin{bemerk}
By theorem \ref{J(u)} and remark \ref{bem M^pm}, each of the above sets in $S\T$ is a Lipschitz graph over $0_\T$ and the Lipschitz constant depends only on $F$.
\end{bemerk}

\begin{lemma}[monotonicity of invariant tori w.r.t. rotation vectors]\label{T_h monotone}
Suppose $\TT_i\subset S\T$ are $\phi^t$-invariant graphs over $0_\T$ for $i=1,...,k \geq 3$ and that there are $h_i\in G_F$, such that each orbit on $\TT_i$ has rotation vector $h_i$. Assume that the $h_i$ are cyclically ordered w.r.t. the orientation of $G_F$. Then in each $S_x\T$ the intersections $S_x\T\cap \TT_i$ have the same cyclic order as the $h_i \in G_F$.
\end{lemma}

\begin{proof}
It is enough to prove the statement for $k=3$, then the general case follows. Let $x\in \R^2$ and $v_i = S_x\R^2\cap \TT_i$, $i=1,2,3$. Since successive intersections of the $c_i := c_{v_i}$ in $\R^2$ are excluded by minimality, the curves $c_i(0,\infty)$ are pairwise disjoint and $c_2(0,\infty)$ is contained in one of the connected components of $U := \R^2-\cup_{i=1,3}c_i[0,\infty)$ that are bounded by $c_1[0,\infty), c_3[0,\infty)$. Putting disjoint open cones $C_i$ around $\R_{>0} h_i$ and observing that $c_i[T,\infty)\subset C_i$ for some large $T$, we find the connected component of $U$ containing $c_2(0,\infty)$ by following the line $x+tv_2 \approx c_2(t) \in U$ for small $t>0$, which proves the claim.
\end{proof}

\begin{prop}\label{invar tori}
If the gap-condition is not fulfilled for the Finsler metric $F$, then all $\TT_h,\TT_h^\pm$ are invariant tori for $\phi^t$ (i.e. $\pi(\TT)=\T$).
\end{prop}

\begin{proof}
For any rational $h\in G_F$ we have a compact $\phi^t$-invariant torus $\TT_h^0$ from definition \ref{def gap}. The set of compact $\phi^t$-invariant sets in $S\T$ is compact w.r.t. the Hausdorff metric, cf. 13.2.1-3 in \cite{katok hasselblatt}. If $h\in G_F$ is arbitrary, choose a monotone rational sequence $h_i\searrow h$ in $G_F$. W.l.o.g. we get a limit set $\TT=\lim \TT_{h_i}^0$ and any $v\in \TT$ is a limit of a sequence $v_i\in \TT_{h_i}^0$. We have $\pi(\TT)=\T$, as for any $x\in\T$ we have some $v_i \in S_x \T \cap\TT_{h_i}^0$ and hence the existence of a $v=\lim v_i\in S_x\T\cap\TT$. Moreover, by the monotonicity in lemma \ref{T_h monotone} and the uniform Lipschitz property of the $\TT_h^0$, the limit set $\TT$ is again a Lipschitz graph over $0_\T$. By construction the $v_i$ are $\eta_i$-semistatics for some sequence $\eta_i\in\F^{h_i}$ and by corollary \ref{alpha C^1} the $\eta_i$ converge monotonically to $\eta_+\in \F^h=[\eta_-,\eta_+]$, so $\TT\subset \NN_{\eta_+}$ by proposition \ref{N semi-cont}. 

In the case where $h$ is irrational we get $\pi(\TT_h) = \pi(\NN_{\nabla\beta(h)}) \supset \pi(\TT) = \T$. For rational $h$ the argument in theorem \ref{rational directions} (iii) shows that for some point $x$ in the gap between two neighboring $q_0,q_1$ from $\M^h$ we obtain a vector $v\in S_x\T\cap \NN_h^+(q_0,q_1)$ as limit of some sequence $v_i\in \TT_{h_i}^0$, i.e. $\TT\cap \NN_h^+(q_0,q_1)\neq \emptyset$. If we also have some $w\in \TT\cap \NN_h^-(q_0,q_1)$, we can follow the flowlines of $v,w$ and w.l.o.g. $\pi(v)=\pi(w)$, as the heteroclinics $c_v,c_w$ always intersect. But since $v\neq w$, this contradicts the graph property of $\TT$. This shows $\TT \subset \TT_h^+$ and hence $\pi(\TT_h^+)\supset \pi(\TT)=\T$. Analogously $\TT_h^-$ projects surjectively.
\end{proof}

We can now prove theorem I announced in the introduction for Finsler metrics. Using proposition \ref{maupertuis}, the theorem carries over to Tonelli Lagrangians and energies above \mane's strict critical value.

\begin{thm}\label{main thm}
If the topological entropy of a Finsler geodesic flow on $\T$ in the unit tangent bundle $S\T$ vanishes, then in $S\T$ there are invariant graphs $\TT_h, \TT_h^\pm\subset S\T$ as in definition \ref{def invar tori} for all $h\in G_F$. If $v\in S\T$ does not lie on one of these invariant graphs, the orbit of $v$ lies in the space between two graphs $\TT_h^-,\TT_h^+$ in $S\T$ of some common rational rotation vector $h$, while these graphs intersect in the periodic minimizers of rotation vector $h$ (the Mather set $\M^h$).
\end{thm}

\begin{proof}
Apply propositions \ref{h_top=0} and \ref{invar tori}, so in $S\T$ there are the invariant graphs $\TT_h,\TT_h^\pm$. By theorems \ref{irrational directions} and \ref{rational directions}, the union of all $\TT_h, \TT_h^\pm$ is $\NN := \cup_{\eta\in H_F} \NN_\eta$ and by proposition \ref{N semi-cont} $\NN$ is a closed set. Let $v\in S_x\T-\NN$ and $v_-,v_+\in S_x\T\cap \NN$ be the closest vectors to $v$, s.th. $v$ is contained in the (oriented) segment $(v_-,v_+)\subset S_x\T$ and let $h_\pm=\rho(c_{v_\pm})\in G_F$. If $h_-\neq h_+$, we could by lemma \ref{T_h monotone} put some semistatic in both parts of $S_x\T-\{v_-,v_+\}$ and get a contradiction. Hence $h_-=h_+=h$ and by the uniqueness of irrational invariant tori in $\NN$, $h$ is rational. Since the part outside of the space between the $\TT_h^\pm$ contains other invariant tori, $v$ is contained in the space between $\TT_h^\pm$. 
\end{proof}

\begin{bemerk} \begin{itemize}
\item[(i)] For rational $h\in G_F$ we have $\TT_h^-=\TT_h^+$ iff $\pi(\M^h)=\T$ iff $\beta$ is differentiable in $h$ (cf. remark \ref{bem rational directions}). This shows that $\beta:H_1(\T,\R)-\{0\}\to\R$ is $C^1$ iff the unit tangent bundle is foliated by invariant tori $\TT_h$ that are Lipschitz graphs over $0_\T$ (sometimes referred to as complete $C^0$-integrability). For results in this direction cf. \cite{massart sorrentino}.

\item[(ii)] For irrational $h\in G_F$ the torus $\TT_h$ coincides with the Mather set $\M^h$, provided $\pi(\TT_h)=\T$, i.e. each geodesic $c_v(\R)$ with $v\in\TT_h$ is dense in $\T$. This is a version of a more general result, cf. theorem 1 in \cite{rocha}, observing that each orbit in $\TT_h$ is homoclinic to $\M^h$.
\end{itemize}\end{bemerk}

%%%%%%%%%%%%%%%%%%%%%%%%%%%%%%%%%%%%%%%%%%%%%%%%%%%%%%%%%%%%%%%%%%%%%%%%%%%%%%%%%%%%
%%%%%%%%%%%%%%%%%%%%%%%%%%%%%%%%%%%%%%%%%%%%%%%%%%%%%%%%%%%%%%%%%%%%%%%%%%%%%%%%%%%%
%%%%%%%%%%%%%%%%%%%%%%%%%%%%%%%%%%% Katok examples %%%%%%%%%%%%%%%%%%%%%%%%%%%%%%%%%
%%%%%%%%%%%%%%%%%%%%%%%%%%%%%%%%%%%%%%%%%%%%%%%%%%%%%%%%%%%%%%%%%%%%%%%%%%%%%%%%%%%%
%%%%%%%%%%%%%%%%%%%%%%%%%%%%%%%%%%%%%%%%%%%%%%%%%%%%%%%%%%%%%%%%%%%%%%%%%%%%%%%%%%%%

\section{Katok's examples, proof of theorem III}\label{section katok}

In this section we consider Riemannian and Finsler metrics on $\R^2$ and the cylinder $\CC = \R/2\pi \Z\times I$, where $I\subset \R$ is an interval. The notation is independent of the notation in the previous sections. $\skp, \norm$ denote the euclidean scalar product and norm on $\R^2$. We work in the Hamiltonian setting, i.e. in $T^*\R^2$ with the canonical symplectic structure, identifying $T^*\R^2$ with $\R^4$ via
\[ \R^4 = \R^2\times\R^2 \ni \xi=(x,\eta) \cong \la \eta,.\ra \in T_x^*\R^2. \]
For $\xi=(x,\eta)\in\R^4, a\in\R$ and a function $g:\R\to\R$ we write
\begin{align*}
g(\xi)=g(x_2), \quad \|\xi\|=\|\eta\|, \quad a \xi=(x, a\eta). 
\end{align*}
We write $X_H,\phi_H^t$ for the Hamiltonian vector field / flow associated to $H:T^*\R^2\to\R$. Recall $\{ H,H' \}=0$ for the Poisson bracket $\{.,.\}$ iff $H'$ is constant along the flow lines of $\phi_H^t$.

\begin{defn}\label{def rotational}
A Riemannian metric $G$ on $\CC = \R/2\pi \Z\times I$ of the form
\[ G_x = g^2(x_2) \cdot \skp, \quad g:I\to(0,\infty) \]
is called a \emph{rotational metric}.
\end{defn}

\begin{bemerk}\label{bem def rotational}\begin{itemize}
\item[(i)] The geodesic flow in $T^*\R^2$ of $G$ can be described by the Hamiltonian flow of the dual Finsler norm
\[ F_g : T^*\R^2\to\R, \quad F_g = \frac{\norm}{g}. \]
Recall that $X_{\frac{1}{2}F_g^2} = F_g \cdot X_{F_g}$, so dual Finsler norms and dual Finsler energies generate Hamiltonian flows which are reparametrisartions of each other, while $F_g$ has the advantage that $\phi_{F_g}^tr\xi=r\phi_{F_g}^t\xi$. $\phi_{F_g}^t$ admits $\xi\mapsto \eta_1$ as an integral and one easily sees that $F_g,\eta_1$ are independent a.e., i.e. the geodesic flow of $G$ is completely integrable. Moreover, by theorem 1 in \cite{paternain}, the topological entropy of $\phi_{F_g}^t$ vanishes.

\item[(ii)] Let $c=(c_1,c_2,c_3):J\to \R^3$ be an immersed $C^\infty$ space curve with $c_1>0$ and $c_2\equiv 0$ and let $A(s) = \left ( \begin{smallmatrix} \cos s & -\sin s & 0 \\ \sin s & \cos s & 0 \\ 0 & 0 & 1 \end{smallmatrix}\right )$ be the rotational matrix about the $x_3$-axis. Set
\[ \varphi:\R\times J \to \R^3, \quad  \varphi(x)=A(x_1)\cdot c(x_2) . \]
Then $\varphi$ defines (locally) a surface of revolution $\Sigma\subset\R^3$ with the induced Riemannian metric $\skp_{\R^3}|_{T\Sigma\times T\Sigma}$. We have
\[ (\varphi^*\skp_{\R^3})_x(v,w)= \la v, G_0(x) w \ra_{\R^2}, \quad G_0(x) :=\begin{pmatrix} c_1^2(x_2) & 0 \\ 0 & \|\dot c(x_2)\|^2 \end{pmatrix}. \]
We solve $h' = c_1\circ h$ for a function $h:I\to J$. Assuming $c$ to be parametrised by euclidean arc length we obtain for $\tilde \varphi(x) = \varphi(x_1,h(x_2))$ that
\[ (\tilde \varphi^*\skp_{\R^3})_x = g^2(x_2) \cdot \skp, \quad g=c_1\circ h \]
on $\R\times I$. Hence the name rotational metric.
\end{itemize}\end{bemerk}

\begin{defn}
Let $G=g^2 \skp$ be a rotational metric, $\al > 0$ a constant and $\psi : T^*\R^2 \to \R$ a function which is positively homogeneous of degree one in the fibres, smooth off the zero-section and such that $\psi$ commutes with $F_g,\eta_1$, i.e.
\[ \{ F_g,\psi \} = \{ \eta_1,\psi \}=0. \]
We call the Hamiltonians
\[ H_{\al,\psi} := \al \cdot F_g + \psi ~ : T^*\R^2\to \R \]
\emph{generalized Katok-Ziller metrics}.
\end{defn}

The functions $H_{\al,\psi}$ are positively homogeneous of degree one and if $\psi$ is small in $C^2$ ensuring that $\frac{1}{2}H_{\al,\psi}^2$ is strictly convex, $H_{\al,\psi}$ defines a dual Finsler norm. Since $H_{\al,\psi}$ still admits the integrals $F_g,\eta_1$, the geodesic flow of $H_{\al,\psi}$ is completely integrable and has $\h(\phi^t_{H_{\al,\psi}})=0$.

Dual Finsler metrics $H_{\al,\psi}$ with $\psi(\eta)=\beta \cdot \eta_1, \beta =\const$ were first studied by Katok \cite{katok}, later by Ziller \cite{ziller}. Katok also considered a function $\psi(\eta)$ like the one that we will encounter in lemma \ref{periodic flow}.

Katok's construction starts with a periodic geodesic flow. To describe this, let $g_0:\R\to(0,\infty)$ be the function obtained from describing $S^2$ as a surface of revolution by a rotational metric as in remark \ref{bem def rotational} (ii), where $\{x_1=0\}$ is assumed to be mapped to the equator in $S^2$. One can check that
\[ g_0(t)=\frac{2e^t}{1+e^{2t}}, \quad g_0:\R\to (0,\infty). \]
Observe that $g_0$ is strictly decreasing in $[0,\infty)$ and $g_0(-t)=g_0(t)$. The flow $\phi_{F_{g_0}}^t$ in $T^*\CC$ for $G=g_0^2\skp$ is periodic with period $2\pi$, where $\CC=\R/2\pi\Z \times\R$.

In the next lemma we study surfaces of revolution $\Sigma\subset\R^3$, such that $\Sigma$ intersects the round sphere $S^2$ in a belt around the equator. At the level of rotational metrics, this means that around $\{x_1=0\}$, $g$ coincides with $g_0$.

\begin{lemma}\label{periodic flow}
Consider a rotational metric $G=g^2 \skp$ on $\R\times I$, where $I\subset \R$ is an interval containing $0$ and assume that
\[ \exists ~ b>0, [-b,b]\subset I : \quad  g|_{[-b,b]} = g_0|_{[-b,b]} . \]
For $a\in (0,b)$ denote by $M_a \subset T^*\R^2$ the set
\[ \left \{ r \cdot \xi = (x,r\cdot \eta) \in T^*\R^2  ~ : ~ r>0 ~  , ~ F_g(\xi) =1 ~ , ~ |x_2|\leq a  ~ , ~ \eta_1 \geq g_0(a) ~ \right \} . \]
Chose $0<a_0<a_1<b$ and two functions $\chi:\R^2\to \R, f:\R\to\R$, where $f$ is smooth, with
\[ \chi(x_1,x_2)= \begin{cases}  1 & : |x_2| \leq b \\ 0 & : |x_2|>b \end{cases} , \qquad f(t) = \begin{cases}  1 & : t \geq g_0(a_0)  \\ 0 & : t \leq g_0(a_1)  \end{cases} \]
and set
\[ \psi: T^*\R^2 \to \R, \quad \psi(x,\eta) := \chi(x) \cdot f(\eta_1/F_g(\eta)) \cdot \eta_1. \]
Then $\psi$ is smooth outside $0_{\R^2}$ and $\{F_g,\eta\}=\{\eta_1,\psi\}=0$. For $|\beta|$ small consider the generalized Katok-Ziller metrics
\[ F_{\al,\beta} : = H_{\al,\beta\cdot \psi} \]
on the cylinder $\CC=\R/2\pi \Z \times I$. Then
\[ \forall t,(\al,\beta): ~ \phi^t_{F_{\al,\beta}}M_{a_0}=M_{a_0}, ~ \phi^t_{F_{\al,\beta}}M_{a_1}=M_{a_1} \]
and
\[ \phi_{F_{1,0}}^{2\pi}|_{M_{a_0}} = \phi_{F_{0,1}}^{2\pi}|_{M_{a_0}}=\id_{M_{a_0}} . \]
\end{lemma}

\begin{bemerk}\begin{itemize}
\item[(i)] $M_a$ can be thought of as a forward cone in each $T_x\R^2$, where for $x_2=0$ the cone is opened the widest and for $|x_2|\nearrow a$ the cone becomes a ray. For $|x_2|>a$ the cone is empty.

\item[(ii)] Obviously the so defined $\psi$ is homogeneous of degree one. Moreover, $\psi$ is independent of $x_1$ and hence $F_{\al,\beta}$ is defined on the cylinder $\CC$.
\end{itemize}\end{bemerk}

\begin{proof}
\underline{Step 1} (invariance of $M_a$ under $\phi_{F_g}^t$ for all $a\in (0,b)$). By $\phi_{F_{\al,\beta}}^t r\xi=r\phi_{F_{\al,\beta}}^t \xi$ for all $(\al,\beta)$ we can restrict ourselves to $\{F_g=1\}$. Let $\xi=(x,\eta)$ with $F_g(\xi)=1$, then $\eta_1^2=g^2(x_2) - \eta_2^2 \leq g^2(x_2)$. Hence, if $|x_2|\in (a,b]$, then $\eta_1 < g_0(a)$, while $\eta_1$ is $\phi_{F_g}^t$-invariant. This shows that if $\xi\in M_a$, then $\phi_{F_g}^t\xi$ cannot leave $\{|x_2|\leq a\}$.

\underline{Step 2} (invariance of $M_{a_0},M_{a_1}$ under $\phi_\psi^t$). Observe that $\psi(\xi)= \eta_1$ for $\xi=(x,\eta)\in M_{a_0}$, so in $M_{a_0}$ we have $\phi^t_\psi(x,\eta)=(x+te_1,\eta)$ and hence invariance of $M_{a_0}$. Now let $r\xi=(x,r\eta)\notin M_{a_1}$ with $F_g(\xi)=1$, then either $|x_2|>a_1$ or $\eta_1< g_0(a_1)$. In the latter case we have $f(r\eta_1/F_g(r\xi))=f(\eta_1)=0$. In the first case it follows that either $\chi=0$ or $|x_2|\in (a_1,b]$. Again in the latter case we obtain $g_0(x_2)<g_0(a_1)$ and hence $|\eta_1|=\sqrt{g_0^2(x_2)-\eta_2^2} < g_0(a_1)$ and again $f=0$. Altogether we find $\psi=0$ outside $M_{a_1}$, so the invariance of the complement of $M_{a_1}$ follows. Moreover, since $\psi=0$ in an open neighborhood of $\{\chi=0\}$, $\psi$ is smooth in $T^*\R^2- 0_{\R^2}$.

\underline{Step 3} ($F_{\al,\beta}$ is a generalized Katok-Ziller metric for small $\beta$). By step 2 we have $\psi=0$ outside $M_{a_1}$, so the equalities $\{ F_g,\psi \} = \{ \eta_1,\psi \}=0$ are trivial. But $\psi$ is locally independent of $\chi$ inside $M_{a_1}$ and here $\psi$ is defined in terms of $\eta_1, F_g$, so the equalities hold also in this case.

\underline{Step 4} (periodicity of the flows $\phi_{F_{1,0}}^t,\phi_{F_{0,1}}^t$). We saw above that in $M_{a_0}$ we have $\psi= \eta_1$. The periodicity of $\phi_{F_{0,1}}^t$ follows. The periodicity of $\phi_{F_{1,0}}^t=\phi_{F_g}^t$ follows, since in the $\phi_{F_g}^t$-invariant set $M_{a_0}$ the Riemannian metric $G$ is the same as the rotational metric $G=g_0\cdot \skp$ obtained from the round sphere $S^2$ and in this case we know that $\phi^t_{F_{g_0}}$ is $2\pi$-periodic.
\end{proof}

We can now readily apply theorem A from \cite{katok} and prove the theorem III stated in the introduction.

\begin{proof}[Proof of theorem III]
\underline{Step 1} (existence of a non-reversible $F_0$). We work in $\CC=\R/2\pi\Z\times I$ with a rotational metric as above, use the same notation as Katok, just writing $F_{\al,\beta}$ instead of $H_{\al,\beta}$, and apply theorem A from \cite{katok}, where we take $M^{2m}$ to be the set $M_{a_0}\subset (T^*\R^2,\om)$ defined in lemma \ref{periodic flow} (the number $a_0$ is defined by the width of the strip around the equator in $\Sigma\cap S^2$) and $(\al_0,\beta_0)=(1,0)$. In particular
\[ M_D = \{ \xi\in M_{a_0}: X_{F_{1,0}}(\xi), X_{F_{0,1}}(\xi) \text{ lin. dependent} \} \]
consists precisely of the equator in the various velocities. Properties (i) and (ii) are proven for $F_0^* := \HH$ in $M_{a_0}=: Z^*$, where $\HH$ is given by Katok's theorem. Moreover $F_0^*$ coincides with a generalized Katok-Ziller metric $F_{\al,\beta}$, as defined in lemma \ref{periodic flow}, together with all its derivatives in $\partial M_{a_0}$ and in a single periodic orbit in $\Int (M_{a_0})$ (the equator, i.e. $M_D$). Here $(\al,\beta)$ is arbitrarily close to $(1,0)$.

To show that $(F_0^*)^2$ is stricly convex, just observe that in Katok's theorem $F_0^*, F_{\al,\beta}$ are close in $C^2$ and that $F_{\al,\beta}^2$ is strictly convex by the smallness of $\beta$. We extend $F_0^*$ to all of $T^*\CC$ using $F_{\al,\beta}$. Finally, property (iii) follows from (ii): the entropy of $F_{\al,\beta}$ vanishes, since the flow is completely integrable. The claim now follows from the general fact that the topological entropy is bounded by the growth of the number of periodic orbits (cf. corollary 4.4 in \cite{katok1}), which is sub-exponential for $F_0^*$ in $M_{a_0}$ by (ii).

\underline{Step 2} (make $F_0$ reversible). Recall from the proof of lemma \ref{periodic flow} that there is the neighborbood $M_{a_1}\supset M_{a_0}$, such that in $U:= T^*\CC-M_{a_1}$ we have $\psi=0$, i.e. here $F_0^*=\al \cdot F_g$, which is just the rescaled Riemannian metric on $\Sigma$. Hence we can define a new dual Finsler metric
\[ F_1^*(x,\eta) := \begin{cases} F_0^*(x,\eta) & : \eta_1\geq 0 \\ F_0^*(x,-\eta) & : \eta_1\leq 0 \end{cases} \]
on $\R\times I$, which is now a reversible dual Finsler metric. The metric is unchanged in the $\phi_{F_0^*}^t$-invariant set $\{\eta_1\geq 0\}$ and in $\{\eta_1\leq 0\}$ the geodesic flow of $F_1^*$ is just the reversed flow of $F_0^*$ from $\{\eta_1\geq 0\}$.
\end{proof}


\begin{thebibliography}{99}

\bibitem{angenent} Sigurd Angenent -- \emph{The topological entropy and invariant circles of an area preserving twistmap}. IMA volumes in mathematics and its applications 44 (1992), 1-7.

\bibitem{bangert} Victor Bangert -- \emph{Mather sets for twist maps and geodesics on tori}. Dynamics Reported 1 (1988), 1-56.

\bibitem{bangert2} Victor Bangert -- \emph{Geodesic rays, Busemann funtions and monotone twist maps}. Calculus of Variations and Partial Differential Equations 2.1 (1994), 49-63.

\bibitem{shen} David Dai-Wai Bao, Shiing Shen Chern, Zhongmin Shen -- \emph{An introduction to Riemann-Finsler geometry} (book). Graduate Texts in Mathematics 200, Springer Verlag (2000).

\bibitem{rabinowitz1} S. V. Bolotin, Paul H. Rabinowitz -- \emph{Some geometrical condition for the existence of chaotic geodesics on a torus}. Ergodic Theory and Dynamical Systems 22.5 (2002), 1407-1428.

\bibitem{bosetto serra} Elena Bosetto, Enrico Serra -- \emph{A variational approach to chaotic dynamics in periodically forced nonlinear oscillators}. Annales de l'Institut Henri \Poincare (C) Non Linear Analysis 17.6, Elsevier Masson (2000).

\bibitem{contreras} Gonzalo Contreras, Renato Iturriaga -- \emph{Global minimizers of autonomous Lagrangians}. Preprint (2000).

\bibitem{contreras1} Gonzalo Contreras, Renato Iturriaga, G.P. Paternain, M. Paternain -- \emph{Lagrangian graphs, minimizing measures and \mane's critical values}. Geom. Funct. Anal. 8 (1998), 788-809.

\bibitem{denzler} Jochen Denzler -- \emph{Mather sets for plane Hamiltonian systems}. Zeitschrift f\"ur Angewandte Mathematik und Physik (ZAMP) 38.6 (1987), 791-812.

\bibitem{fathi} Albert Fathi -- \emph{Weak KAM theorem in Lagrangian dynamics, preliminary version number 10}. Preprint (2008).

\bibitem{glasmachers1} Eva Glasmachers, Gerhard Knieper -- \emph{Minimal geodesic foliation on $\T$ in case of vanishing topological entropy}. arXiv:1101.1660 [math.DS] (2011).

\bibitem{hedlund} Gustav A. Hedlund -- \emph{Geodesics on a two-dimensional Riemannian manifold with periodic coefficients}. The Annals of Mathematics 33.4 (1932) 719-739.

\bibitem{katok} Anatole Katok -- \emph{Ergodic perturbations of degenerate integrable Hamiltonian systems}. Mathematics of the USSR-Izvestiya 7.3 (1973), 535.

\bibitem{katok1} Anatole Katok -- \emph{Lyapunov exponents, entropy and periodic orbits for diffeomorphisms}. Publications mathmatiques de l'I.H.E.S. 51 (1980), 137-173.

\bibitem{katok hasselblatt} Anatole Katok, Boris Hasselblatt -- \emph{Introduction to the modern theory of dynamical systems} (book). Encyclopedia of Mathematics and its Applications 54, Cambridge University Press (1996).

\bibitem{massart sorrentino} Daniel Massart, Alfonso Sorrentino -- \emph{Differentiability of Mather's average action and integrability on closed surfaces}. Nonlinearity 24.6 (2011), 1777-1793.

\bibitem{mather} John N. Mather -- \emph{Action minimizing measures for positive definite Lagrangian systems}. Mathematische Zeitschrift 207.1 (1991), 169-207.

\bibitem{mather1} John N. Mather -- \emph{Differentiability of the minimal average action as a function of the rotation number}. Boletim da Sociedade Brasileira de Matematica-Bulletin/ Brazilian Mathematical Society 21.1 (1990), 59-70.

\bibitem{morse} Harold Marston Morse -- \emph{A fundamental class of geodesics on any closed surface of genus greater than one}. Transactions of the American Mathematical Society 26.1 (1924), 25-60.

\bibitem{moser1} J\"urgen Moser -- \emph{Monotone twist mappings and the calculus of variations}. Ergodic Theory and Dynamical Systems 6.3 (1986), 401-413.

\bibitem{paternain} Gabriel Paternain -- \emph{Entropy and completely integrable Hamiltonian systems}. Proceedings of the American Mathematical Society 113.3 (1991).

\bibitem{rabinowitz} Paul H. Rabinowitz -- \emph{Heteroclinics for a reversible Hamiltonian system}. Ergodic Theory and Dynamical Systems 14.4 (1994), 817-829.

\bibitem{rabinowitz2} Paul H. Rabinowitz -- \emph{The calculus of variations and the forced pendulum}. Hamiltonian Dynamical Systems and Applications, Springer Netherlands (2008), 367-390.

\bibitem{rocha} Alexandre Rocha, Mario J. D. Carneiro -- \emph{A dynamical condition for differentiability of Mather's average action}. arXiv:1208.1474v1 [math.DS] (2012).

\bibitem{rockafellar} R. Tyrell Rockafellar -- \emph{Convex Analysis} (book). Princeton Landmarks in Mathematics and Physics 28, Princeton University Press (1997).

\bibitem{schroeder} Jan Philipp Schr\"oder -- \emph{Tonelli Lagrangian systems on the 2-torus with vanishing topological entropy}. Ph.D. thesis, Ruhr-Universit\"at Bochum (in preparation).

\bibitem{sorrentino} Alfonso Sorrentino -- \emph{Lecture notes on Mather's theory for Lagrangian systems}. arXiv: 1011.0590 [math.DS] (2010).

\bibitem{walters} Peter Walters -- \emph{An introduction to ergodic theory} (book). Graduate Texts in Mathematics 79, Springer Verlag (2000).

\bibitem{zaustinsky} Eugene M. Zaustinsky -- \emph{Extremals on compact $E$-surfaces}. Transactions of the American Mathematical Society 102.3 (1962), 433-445.

\bibitem{ziller} Wolfgang Ziller -- \emph{Geometry of the Katok examples}. Ergodic Theory and Dynamical Systems 3 (1982), 135-157.


\end{thebibliography}
\end{document}